\documentclass[12pt,reqno]{amsart}

\usepackage{amsfonts}
\usepackage{eurosym}
\usepackage{amssymb}
\usepackage{amsthm}
\usepackage{amsmath}
\usepackage{amsaddr}
\usepackage{bm}
\usepackage{cite}
\usepackage{mathrsfs}
\usepackage{xcolor}
\usepackage[OT1]{fontenc}
\usepackage[left=1.8cm, right=1.8cm, top=3cm]{geometry}
\usepackage{hyperref}
\hypersetup{colorlinks=true, linkcolor=blue, citecolor=red}

\RequirePackage{times}
\allowdisplaybreaks
\newtheorem{theorem}{Theorem}[section]

\newtheorem{remark}[theorem]{Remark}

\def\n{\textbf{\textit{n}}}
\def\R3{\mathbb{R}^3}
\def\F2o{\overline{F_2}}

\def\d{{\rm d}}
\def \l {\langle}
\def \r {\rangle}
\def\LL{{\mathbf{L}}}
\def\V{{\mathbf{V}}}

\def\H{\mathbf{H}}

\def\A{\mathbf{A}}

\def\f{\textbf{\textit{f}}}

\def\D {{\mathbb D}}

\def\uu{\textbf{\textit{u}}}
\def\vv{\textbf{\textit{v}}}
\def\ww{\textbf{\textit{w}}}
\def\ddt{\frac{\d}{\d t}}

\def\J{\mathbf{J}}

\def\P{\mathbb{P}}
\def\div{\mathrm{div}\,}
\def\L2{L^2(\Omega)}

\DeclareMathOperator*{\esssup}{ess\,sup}

\def \au {\rm}
\def \ti {\it}
\def \jou {\rm}
\def \bk {\it}
\def \no#1#2#3 {{\bf #1} (#3), #2.}
\def \eds#1#2#3 {#1, #2, #3.}
\def \nome#1#2 {{\bf #1}, (#2).}

\begin{document}
\title[The Abels-Garcke-Gr\"{u}n model in 3D]
{Existence and Stability of Strong Solutions to \\ the Abels-Garcke-Gr\"{u}n model in Three Dimensions}
%{On the Abels-Garcke-Gr\"{u}n model in three dimensions \\ for incompressible two-phase flows}
\author[A. Giorgini ]{Andrea Giorgini}

\address{Department of Mathematics\\
Imperial College London\\
London, SW7 2AZ, UK}
\email{a.giorgini@imperial.ac.uk}
\date{\today}



\begin{abstract}
This work is devoted to the analysis of strong solutions to the Abels-Garcke-Gr\"{u}n (AGG) model in three dimensions. First, we prove the existence of local-in-time strong solutions originating from an initial datum $(\uu_0, \phi_0)\in  \H^1_\sigma \times H^2(\Omega)$ such that $\mu_0 \in H^1(\Omega)$ and $|\overline{\phi_0}|\leq 1$. For the subclass of initial data that are strictly separated from the pure phases, the corresponding strong solutions are locally unique. Finally, we show a stability estimate between the solutions to the AGG model and the model H. 
These results extend the analysis achieved by the author in {\it Calc. Var. (2021) 60:100} to three dimensional bounded domains.
\end{abstract}

\maketitle

\section{Introduction}
Given a domain $\Omega \subset \mathbb{R}^3$, we study the Abels-Garcke-Gr\"{u}n (AGG) model in $\Omega \times (0,T)$
\begin{equation}
 \label{AGG}
\begin{cases}
\partial_t (\rho(\phi)\uu) + \div \big( \uu \otimes (\rho(\phi) \uu + \widetilde{\J})\big) - \div (\nu(\phi) \D \uu) + \nabla P= - \div(\nabla \phi \otimes \nabla \phi)\\
\div \uu=0\\
\partial_t \phi +\uu\cdot \nabla \phi = \Delta \mu\\
\mu= -\Delta \phi+\Psi'(\phi),
\end{cases}
%\text{in } \Omega \times (0,T),
\end{equation}
completed with the following boundary and initial conditions
\begin{equation}
\label{AGG-bc}
\begin{cases}
\uu=\mathbf{0}, \quad \partial_\n\phi=\partial_\n \mu=0 \quad &\text{on }  \partial \Omega \times (0,T),\\
\uu(\cdot, 0)=\uu_0, \quad \phi(\cdot, 0)=\phi_0 \quad &\text{in } \Omega.
\end{cases}
\end{equation}
Here, $\n$ is the unit outward normal vector on $\partial \Omega$, and
 $\partial_\n$ denotes the outer normal derivative on $\partial \Omega$.
In the system, $\uu=\uu(x,t)$ represents the volume averaged velocity, $P=P(x,t)$ is the pressure of the mixture, and $\phi=\phi(x,t)$ is the difference of the fluids concentrations. The operator $\D$ is the symmetric gradient $\frac12 (\nabla +\nabla^T)$ . The flux term $\widetilde{\J}$, the density $\rho$ and the viscosity $\nu$ of the mixture are defined as
\begin{equation}
\label{Jrhonu}
\widetilde{\J}= -\frac{\rho_1-\rho_2}{2}\nabla \mu, \quad \rho(\phi)= \rho_1 \frac{1+\phi}{2}+ \rho_2 \frac{1-\phi}{2}, \quad
\nu(\phi)=\nu_1 \frac{1+\phi}{2}+ \nu_2 \frac{1-\phi}{2},
\end{equation}
where $\rho_1$, $\rho_2$ and $\nu_1$, $\nu_2$ are the positive homogeneous density and viscosity parameters of the two fluids.  
The homogeneous free energy density $\Psi$ is the Flory-Huggins potential 
\begin{equation}  
\label{Log}
\Psi(s)=F(s)-\frac{\theta_0}{2}s^2=\frac{\theta}{2}\bigg[ (1+s)\log(1+s)+(1-s)\log(1-s)\bigg]-\frac{%
\theta_0}{2} s^2, \quad s \in [-1,1],
\end{equation}
where the constant parameters $\theta$ and $\theta_0$ fulfill the conditions $0<\theta<\theta_0$. In the sequel, we will often use the non-conservative form of \eqref{AGG}$_1$ 
\begin{equation}
\label{NS-nc}
\rho(\phi) \partial_t \uu + \rho(\phi) (\uu \cdot \nabla)\uu
-\rho'(\phi) (\nabla \mu\cdot \nabla) \uu - \div (\nu(\phi)\D\uu) + \nabla P= - \div(\nabla \phi \otimes \nabla \phi).
\end{equation}
We also recall the total energy associated to system \eqref{AGG} given by
$$
E(\uu,\phi)= E_{\text{kin}}(\uu, \phi) + E_{\text{free}}(\phi)= 
\int_{\Omega}  \frac12 \rho(\phi) |\uu|^2 \, \d x + \int_{\Omega} \frac12 |\nabla \phi|^2 + \Psi(\phi)   \, \d x,
$$
and the corresponding energy equation that reads as
 \begin{equation}
\label{EE}
\ddt E(\uu, \phi) +\int_{\Omega} \nu(\phi) |\D \uu|^2 \, \d x+ 
\int_{\Omega} |\nabla \mu|^2 \, \d x =0.
\end{equation}

The AGG system is a primary model in the theory of diffuse interface (phase field) modeling, which describes the motion of two viscous incompressible fluids with different densities. It was proposed in the seminal work \cite{AGG2012} (see also \cite{AG2018}). The well-known model H is recovered from \eqref{AGG} in the case of matched densities $\rho_1=\rho_2$ (see \cite{GPV1996} for the derivation and \cite{A2009, GMT2019} for the analysis of the model H). The existence of global weak solutions (with finite energy) to the AGG model \eqref{AGG}-\eqref{AGG-bc} has been established in the case of non-degenerate mobility in \cite{ADG2013} and in the case of degenerate mobility in \cite{ADG2013-2}. Global weak solutions were also  proven for viscous non-Newtonian fluids in \cite{AB2018} and in the case of dynamic boundary conditions describing moving contact lines in  \cite{GalGW2019}. Further generalizations to nonlocal versions of the AGG model have been studied in \cite{AT2020} for fractional free energies and in \cite{Frigeri2016} and \cite{Frigeri2020} for free energy with regular convolution kernels. More recently, the existence and uniqueness of regular solutions have been studied in \cite{AW2020} and \cite{G2021}. In \cite{AW2020}, the local well-posedness of strong solutions is proven in three dimensions for polynomial-like potentials $\Psi$ provided that $\uu_0 \in \H^1_\sigma$ and $\phi_0 \in (L^p(\Omega),W^{4}_{p,N}(\Omega))_{1-\frac{1}{p},p}$ for $4<p<6$ (in this range of $p$, $\phi_0 \in H^3(\Omega)$) such that 
$\| \phi_0\|_{L^\infty}\leq 1$. It is worth mentioning that the solution in \cite{AW2020} may not satisfy $| \phi(x,t)|\leq 1$ for all positive times. In \cite{G2021}, the local well-posedness of strong solutions in two dimensional bounded domains has been achieved for the logarithmic potential \eqref{Log} case with initial conditions $(\uu_0, \phi_0)\in  \H^1_\sigma \times H^2(\Omega)$ such that $\mu_0 \in H^1(\Omega)$ and $|\overline{\phi_0}|\leq 1$. In this case, the solution satisfies the physical bound $| \phi(x,t)|\leq 1$ for all times. In addition, in the case of periodic boundary conditions, the strong solutions are shown to be globally defined in time in \cite{G2021}. We also refer the interested reader to \cite{B2002, DSS2007, GKL2018, HMR2012, LT1998, SSBZ2017} and \cite{A2009-2, A2012, AF2008, B2001, CFMMPP2019, GT2020, KZ2015} for the modeling and the analysis of different diffuse interface models with unmatched densities.

The purpose of the present contribution is to study the well-posedness of strong solutions to the AGG model \eqref{AGG}-\eqref{AGG-bc} in bounded domains in $\mathbb{R}^3$. In particular, we aim at generalizing the analysis obtained in \cite{G2021} to the three dimensional case. The first result regarding the existence and uniqueness of strong solutions reads as follows.
\begin{theorem}
\label{mr1}
Let $\Omega$ be a bounded domain of class $C^3$ in $\mathbb{R}^3$. Assume that $\uu_0 \in \H^1_\sigma$ and $\phi_0 \in H^2(\Omega)$ such that $\| \phi_0\|_{L^\infty}\leq 1$, $|\overline{\phi_0}|<1$, $\mu_0= -\Delta \phi_0+ \Psi'(\phi_0) \in H^1(\Omega)$, and $\partial_\n \phi_0=0$ on $\partial \Omega$. Then, there exist $T_0>0$, depending on the norms of the initial data, and (at least) a strong solution $(\uu, P, \phi)$ to system 
\eqref{AGG}-\eqref{AGG-bc} on $(0,T_0)$ in the following sense:
\begin{itemize}
\item[(i)] The solution $(\uu, P, \phi)$ satisfies the properties
\begin{align}
\label{reg-SS}
\begin{split}
&\uu \in C([0,T_0]; \H^1_\sigma) \cap L^2(0,T_0;\H^2_\sigma)\cap W^{1,2}(0,T_0;\LL^2_\sigma),\quad P \in L^2(0,T_0;H^1(\Omega)),\\
&\phi \in L^\infty(0,T_0;W^{2,6}(\Omega)), \
\partial_t \phi \in L^\infty(0,T_0;(H^1(\Omega))')\cap L^2(0,T_0;H^1(\Omega)),\\
&\phi \in L^\infty(\Omega\times (0,T_0)) : |\phi(x,t)|<1 \ \text{a.e. in } \  \Omega\times(0,T_0),\\
&\mu \in L^\infty(0,T_0;H^1(\Omega))\cap L^2(0,T_0;H^3(\Omega)),
\quad F'(\phi) \in L^\infty(0,T_0;L^6(\Omega)).
\end{split}
\end{align}

\item[(ii)] The solution $(\uu, P, \phi)$ fulfills the system \eqref{AGG} almost everywhere in $\Omega \times (0,T_0)$ and the boundary conditions $\partial_\n \phi=\partial_\n \mu=0$ almost everywhere in $\partial \Omega \times (0,T_0)$. 
\end{itemize}
Furthermore, if additionally $\| \phi_0\|_{L^\infty}=1-\delta_0$, for some $\delta_0>0$, then the solution is locally unique. This is, there exists a time $T_1: 0<T_1<T_0$, depending only on the norm of the initial data and $\delta_0$, such that the solution is unique on the time interval $[0,T_1)$.
\end{theorem}

Before proceeding with our second result, it is worth mentioning that the proof of Theorem \ref{mr1}, although still based on a semi-Galerkin approximation,  differs from the one of \cite[Theorem 3.1]{G2021} for several aspects. First, the proof of \cite[Theorem 3.1]{G2021} exploited the continuity of the chemical potential and the regularity of its time derivative, which are properties available for the strong solutions of the convective Cahn-Hilliard equation in two dimensions. Since these are still an open question in three dimensions, we overcome this issue by employing an approximation procedure involving the convective viscous Cahn-Hilliard equation (see Appendix \ref{App-0}), together with an appropriate regularization of the initial datum. Such approximations are crucial to rigorously justify the higher-order Sobolev estimates obtained for the approximate solutions. Secondly, due to the lack of global-in-time separation property in three dimensions, we show local uniqueness of solutions departing from a subclass of initial data such that $\| \phi_0\|_{L^\infty}<1$. For such class of solutions, the separation property holds on a (possible short) time interval by embedding in 
H\"{o}lder spaces. Notice that the argument proposed in \cite{GMT2019} based on estimates in dual spaces cannot be used due to the non-constant density. Moreover, the separation property (or, at least, $L^p$-estimates of $\Psi''(\phi)$ and $\Psi'''(\phi)$) seems to be necessary to control the additional term $\rho'(\phi)(\nabla \mu \cdot \nabla) \uu$. Furthermore, the proof of the uniqueness relies on estimates of higher-order Sobolev spaces compared to the argument in \cite[Theorem 3.1]{G2021}, which is due to the above mentioned novel term $\rho'(\phi)(\nabla \mu \cdot \nabla) \uu$ in \eqref{AGG}$_1$.

Next, we prove a stability result between the strong solutions to the AGG model and the model H departing from the same initial datum in terms of the density values.
\begin{theorem}
\label{stab}
Let $\Omega$ be a bounded domain of class $C^3$ in $\mathbb{R}^3$. Given an initial datum $(\uu_0, \phi_0)$ as in Theorems \ref{mr1}, we consider the strong solution $(\uu, P, \phi)$ to the AGG model with density \eqref{Jrhonu} and the strong solution $(\uu_H, P_H, \phi_H)$ to the model H with constant density $\overline{\rho}>0$, both defined on $[0,T_0]$. Then, there exists a constant $C$, that depends on the norm of the initial data, the time $T_0$ and the parameters of the systems, such that 
\begin{equation}
\label{StaE}
\sup_{t\in [0,T_0]} \| \uu(t)-\uu_H(t)\|_{(\H^1_\sigma)'}
+ \sup_{t\in [0,T_0]} \| \phi(t)-\phi_H(t) \|_{(H^1)'}
\leq  C \Big(  \Big| \frac{\rho_1-\rho_2}{2}\Big| + \Big| \frac{\rho_1+\rho_2}{2}- \overline{\rho}\Big| \Big).
\end{equation}
\end{theorem}
\begin{remark}
Assuming that $\rho_1=\overline{\rho}$ and $\rho_2=\overline{\rho}+\varepsilon$, for (small) $\varepsilon>0$, the stability estimate \eqref{StaE} reads as
$$
\sup_{t\in [0,T_0]} \| \uu(t)-\uu_H(t)\|_{(\H^1_\sigma)'}
+ \sup_{t\in [0,T_0]} \| \phi(t)-\phi_H(t) \|_{(H^1)'}
\leq  C \varepsilon.
$$
\end{remark}

Theorem \ref{stab} justifies the model H as the constant density approximation of the AGG model when the two viscous fluids have negligible densities difference. To make a comparison with \cite[Theorem 3.5]{G2021}, we notice that the estimate holds in dual Sobolev spaces. Indeed, the main idea is to write the momentum equation for the solutions difference $(\uu-\uu_H, \phi-\phi_H)$ as Navier-Stokes equations with constant density and exploit the uniqueness argument introduced in \cite{GMT2019}.
\medskip

\noindent
\textbf{Plan of the paper.} We report in Section \ref{Sec-2} the preliminaries for the analysis. Sections \ref{Loc-Ex} and \ref{Uni} are devoted to the proof of Theorem \ref{mr1}, in particular, the local existence of strong solutions and their uniqueness, respectively. In Section \ref{Sta} we prove the stability result contained in Theorem \ref{stab}. The Appendix \ref{App-0} is concerned with well-posedness results for the convective Viscous Cahn-Hilliard equation.

\section{Notation and Functional Spaces}
\label{Sec-2}
\setcounter{equation}{0}

Let $X$ be a real Banach space. Its norm is denoted by $\|\cdot\|_{X}$ and
the symbol $\langle\cdot, \cdot\rangle_{X',X}$ stands for the duality between $X$ and its dual space $X'$. 
We assume that $\Omega$ is a bounded domain in $\mathbb{R}^3$ with boundary $\partial \Omega$ of class $C^3$. For $p\in [1,\infty]$,
let $L^p(\Omega)$ denote the Lebesgue space with norm $\|\cdot\|_{L^p}$. The inner product in $L^2(\Omega)$ is denoted by $(\cdot, \cdot)$.
For $s \in \mathbb{N}$, $p \in [1,\infty]$, $W^{s,p}(\Omega)$
is the Sobolev space with norm $\|\cdot\|_{W^{s,p}}$. If $p=2$, we use the notation $W^{s,p}(\Omega)=H^s(\Omega)$. 
For every $f\in (H^1(\Omega))'$, we denote by $\overline{f}$ the generalized mean value over $\Omega$ defined by
$\overline{f}=|\Omega|^{-1}\l f,1\r$. If $f\in L^1(\Omega)$, then $\overline{f}=|\Omega|^{-1}\int_\Omega f \, \d x$.
By the generalized Poincar\'{e} inequality, there exists a positive constant $C$ such that 
\begin{equation}
\label{normH1-2}
\| f\|_{H^1}\leq C \big(\| \nabla f\|_{L^2}^2+ |\overline{f}|^2\big)^\frac12, \quad \forall \, f \in H^1(\Omega).
\end{equation}
We recall the Ladyzhenskaya, Agmon and Gagliardo-Nirenberg inequalities in three dimensions
\begin{align}
\label{LADY3}
&\| f\|_{L^3}\leq C \|f\|_{L^2}^{\frac12}\|f\|_{H^1}^{\frac12},  &&\forall \, f \in H^1(\Omega),\\
%\label{LADY4}
%&\| f\|_{L^4(\Omega)}\leq C \|f\|_{L^2(\Omega)}^{\frac14}\|f\|_{H^1(\Omega)}^{\frac34},  &&\forall \, f \in H^1(\Omega),\\
\label{Agmon3d}
&\| f\|_{L^\infty }\leq C \|f\|_{H^1}^{\frac12}\|f\|_{H^2}^{\frac12},  && \forall \, f \in H^2(\Omega),\\
\label{GN-L4}
&\| \nabla f\|_{L^4}\leq C\| f \|_{L^\infty}^\frac12 \| f \|_{H^2}^\frac12, && \forall \, f \in H^2(\Omega),\\
\label{GN-W6}
&\| f\|_{W^{1,4}}\leq C\| f \|_{H^1}^\frac58 \| f \|_{W^{2,6}}^\frac38, && \forall \, f \in W^{2,6}(\Omega).
\end{align}
Next, we introduce the Hilbert spaces of solenoidal vector-valued functions.
In the case of a bounded domain $\Omega \subset \mathbb{R}^3$, we define
\begin{align*}
&\LL^2_\sigma=\{ \uu\in \mathbf{L}^2(\Omega): \mathrm{div}\, \uu=0 \ \text{in } \Omega,\ \uu\cdot \n =0\ \text{on}\ \partial \Omega\},\\
& \H^1_\sigma =\{ \uu\in \mathbf{H}^1(\Omega): \mathrm{div}\, \uu=0 \ \text{in } \Omega,\ \uu=\mathbf{0}\  \text{on}\ \partial \Omega\}.
\end{align*}
We also use $( \cdot ,\cdot )$ and 
$\| \cdot \|_{L^2}$ for
the inner product and the norm in $\LL^2_\sigma$. The space $\H^1_\sigma$ is endowed with the inner product and norm
$( \uu,\vv )_{\H^1_\sigma}=
( \nabla \uu,\nabla \vv )$ and  $\|\uu\|_{\H^1_\sigma}=\| \nabla \uu\|_{L^2}$, respectively.
We report the Korn inequality
\begin{equation}
\label{KORN}
\|\nabla\uu\|_{L^2} \leq \sqrt2\|\D\uu\|_{L^2}, \quad \forall \, \uu \in \H^1_\sigma,
\end{equation}
which implies that $\| \D \uu\|_{L^2}$ is a norm on $\H^1_\sigma$ equivalent to $\| \uu\|_{\H^1_\sigma}$.
We introduce the space
$\H^2_\sigma= \mathbf{H}^2(\Omega)\cap \H^1_\sigma$
with inner product $( \uu,\vv)_{\H^2_\sigma}=( \A\uu, \A \vv )$ and norm $\| \uu\|_{\H^2_\sigma}=\|\A \uu \|_{L^2}$, where $\A=\mathbb{P}(-\Delta)$ is the Stokes operator and $\mathbb{P}$ is the Leray projection from $\mathbf{L}^2(\Omega)$ onto $\LL^2_\sigma$. We recall that there exists a positive constant $C>0$ such that
\begin{equation}
\label{H2equiv}
 \| \uu\|_{H^2}\leq C\| \uu\|_{\H^2_\sigma}, \quad \forall \, \uu\in \H^2_\sigma.
\end{equation}
We denote by $\A^{-1}: (\H^1_\sigma)' \rightarrow \H^1_\sigma$ the inverse map of the Stokes operator. That is,  given $\f \in (\H^!_\sigma)'$, there exists a unique $\uu=\A^{-1} \f \in \H^1_\sigma$ such that
$( \nabla \A^{-1} \f, \nabla \vv )=\l \f, \vv\r$, for all $\vv \in \H^1_\sigma$.
As a consequence, it follows that 
$
\| \f \|_{\sharp}:= \| \nabla \A^{-1} \f \|=\l \f,\A^{-1} \f \r^{\frac12}
$ 
is an equivalent norm on $(\H^1_\sigma)'$. 

Throughout this paper, we will use the symbol $C$ to denote a generic positive constant whose value may change from line to line. The specific value depends on the domain $\Omega$ and the parameters of the system, such as $\rho_\ast$, $\rho^\ast$, $\nu_\ast$, $\nu^\ast$, $\theta$ and $\theta_0$. Further dependencies will be  specified when necessary.

\section{Proof of Theorem \ref{mr1}. Part one: Existence of Solutions}
\label{Loc-Ex}
\setcounter{equation}{0}

In the sequel we will use the following notation
$$
\rho_\ast=\min \lbrace \rho_1,\rho_2\rbrace,
\quad \rho^\ast=\max \lbrace \rho_1,\rho_2\rbrace,
\quad \nu_\ast =\min \lbrace \nu_1,\nu_2 \rbrace,
\quad \nu^\ast =\max \lbrace \nu_1,\nu_2 \rbrace.
$$

\subsection{Approximation of the Initial Datum}
First of all, we approximate the initial concentration $\phi_0$ following the argument introduced in \cite{GMT2019}. For $k\in \mathbb{N}$, there exists a sequence of functions $(\phi_{0,k}, \widetilde{\mu}_{0,k})$ such that 
\begin{equation}
\label{ellapp}
\begin{cases}
-\Delta \phi_{0,k}+F'(\phi_{0,k})=\widetilde{\mu}_{0,k}\quad &\text{ in }\Omega,\\
\partial_\n \phi_{0,k}=0 \quad &\text{ on }\partial \Omega,
\end{cases}
\end{equation}
where $\widetilde{\mu}_{0,k}=h_k \circ \widetilde{\mu}_0$, $h_k$ is a cut-off function and  $\widetilde{\mu}_0=-\Delta \phi_0+ F'(\phi_0)$.
It follows that $\widetilde{\mu}_0\in H^1(\Omega)$, 
and
\begin{equation}
\label{mukH1}
\|\widetilde{\mu}_{0,k}\|_{H^1}\leq \| \widetilde{\mu}_0\|_{H^1}.
\end{equation}
There exists a unique solution $\phi_{0,k}$ to \eqref{ellapp} such that
$\phi_{0,k}\in H^2(\Omega)$, 
$F'(\phi_{0,k}) \in L^2(\Omega)$, 
which satisfies \eqref{ellapp} almost everywhere in $\Omega$ and
$\partial_\n \phi_{0,k}=0$ almost everywhere on $\partial \Omega$.
In addition, there exist $\widetilde{m}\in (0,1)$, which is independent of $k$, and $\overline{k}$ sufficiently large such that 
\begin{equation}
\label{psiH1}
\| \phi_{0,k}\|_{H^1} \leq 1+\| \phi_0\|_{H^1},
\quad
|\overline{\phi_{0,k}}|\leq \widetilde{m}<1, 
\quad 
\| \phi_{0,k}\|_{H^2}\leq C(1+\| \widetilde{\mu}_{0}\|),
\quad \forall \, k> \overline{k}.
\end{equation}
Furthermore, since
$$
\| F'(\phi_{0,k})\|_{L^{\infty}}
\leq \| \widetilde{\mu}_{0,k}\|_{L^{\infty}}\leq k.
$$
As a byproduct, there 
exists $\delta=\delta(k)>0$ such that
\begin{equation}
\label{psikinf}
\| \phi_{0,k}\|_{L^\infty}\leq 1-\delta.
\end{equation}
As a consequence, due to $F'(\phi_{0,k})\in H^1(\Omega)$, it is easily seen that $\phi_{0,k}\in H^3(\Omega)$. 
Finally, observing that $\widetilde{\mu}_{0,k}\rightarrow \widetilde{\mu_0}$ in $L^2(\Omega)$, it follows that
$\phi_{0,k} \rightarrow \phi_0$ in $H^1(\Omega)$. 

\subsection{Definition of the Approximate Problem}
Let us consider the family of eigenfunctions  $\lbrace \ww_j\rbrace_{j=1}^\infty$ and eigenvalues $\lbrace \lambda_j\rbrace_{j=1}^\infty$ of the Stokes operator $\A$. For any integer $m\geq 1$, let $\V_m$ denote the finite-dimensional subspaces of $\LL^2_\sigma$ defined as $\V_m= \text{span}\lbrace \ww_1,...,\ww_m\rbrace$. The finite-dimensional spaces $\V_m$ are endowed with the norm of $\LL^2_\sigma$.
The orthogonal projection on $\V_m$ with respect to the inner product in $\LL^2_\sigma$ is denoted by $\P_m$. Recalling that $\Omega$ is of class $C^3$, the regularity theory of the Stokes operator yields that $\ww_j \in \mathbf{H}^3(\Omega)\cap \H^1_\sigma$ for all $j\in \mathbb{N}$. As a consequence, the following inverse Sobolev embedding inequalities hold for all $\vv \in \V_m$
\begin{equation}
\label{Rev-SI}
\| \vv\|_{H^1}\leq C_m \| \vv\|_{L^2},\quad
\| \vv\|_{H^2}\leq C_m \| \vv\|_{L^2},\quad
\| \vv \|_{H^3}\leq C_m \| \vv\|_{L^2}.
\end{equation}

Let us set $T>0$. For any $k>0, \alpha \in (0,1)$ and $m \in \mathbb{N}$, we claim that there exists an approximate solution $(\uu_m, \phi_m)$ to the system \eqref{AGG} -\eqref{AGG-bc} in the following sense:
\begin{align}
\label{reg-AS}
\begin{split}
&\uu_m \in C^1([0,T];\V_m),\\
&\phi_m \in L^\infty(0,T;H^3(\Omega)), \
\partial_t \phi_m \in L^\infty(0,T;H^1(\Omega))\cap L^2(0,T;H^2(\Omega)),\\
&\phi_m \in L^\infty(\Omega\times (0,T)) : |\phi_m(x,t)| \leq 1-\delta \ \text{a.e. in } \  \Omega\times(0,T),\\
&\mu_m \in L^\infty(0,T;H^2(\Omega)) \cap W^{1,2}(0,T;L^2(\Omega)),
\end{split}
\end{align}
for some $\delta >0$, such that
\begin{align}
\label{NS-w}
\begin{split}
&\begin{aligned}[t]
(\rho(\phi_m) \partial_t \uu_m, \ww)&+(\rho(\phi_m)(\uu_m\cdot \nabla)\uu_m,\ww)+(\nu(\phi_m) \D \uu_m, \nabla \ww)\\
&-\frac{\rho_1-\rho_2}{2} ( (\nabla \mu_m \cdot \nabla) \uu_m, \ww)= (\mu_m \nabla \phi_m, \ww), 
\end{aligned}
\end{split}
\end{align}
for all $\ww \in \V_m$ and $t \in [0,T]$, 
\begin{equation}
\label{CH-w}
\partial_t \phi_m +\uu_m \cdot \nabla \phi_m = \Delta \mu_m, \quad
\mu_m= \alpha \partial_t \phi_m -\Delta \phi_m+\Psi'(\phi_m)
\quad \text{a.e. in } \ \Omega \times (0,T),
\end{equation}
together with
\begin{equation}
\label{bc-AS}
\begin{cases}
\uu_m=\mathbf{0}, \quad \partial_\n \phi_m=\partial_\n \mu_m=0 \quad &\text{on } \partial \Omega \times (0,T),\\
\uu_m(\cdot,0)=\P_m \uu_{0}, \ \phi(\cdot,0)=\phi_{0,k} \quad &\text{in } \Omega.
\end{cases}
\end{equation}

\subsection{Existence of Approximate Solutions}
We exploit a fixed point argument to show the existence of $(\uu_m, \phi_m)$ satisfying \eqref{reg-AS}-\eqref{bc-AS}.
For this purpose, we fix $\vv \in W^{1,2}(0,T;\V_m)$.
We consider the convective Viscous Cahn-Hilliard system
\begin{equation}
\label{cCH}
\begin{cases}
\partial_t \phi_m +\vv \cdot \nabla \phi_m = \Delta \mu_m\\
\mu_m= \alpha \partial_t \phi_m -\Delta \phi_m+F'(\phi_m)-\theta_0 \phi_m
\end{cases}
\quad \text{in }\ \Omega \times (0,T),
\end{equation}
which is equipped with the boundary and initial conditions
\begin{equation}
\label{cCH-c}
\partial_\n \phi_m=\partial_\n \mu_m=0 \quad \text{on} \ \partial \Omega \times (0,T), \quad \phi_m(\cdot, 0)=\phi_{0,k} \quad \text{in }\ \Omega.
\end{equation}
Thanks to Theorem \ref{r-vCH}, there exists a unique solution $\phi_m$ to \eqref{cCH}-\eqref{cCH-c} such that
\begin{equation}
\label{phi-As}
\begin{split}
&\phi_m \in L^\infty(0,T;H^3(\Omega)), \quad 
\partial_t \phi_m \in L^\infty(0,T; H^1(\Omega))\cap L^2(0,T;H^2(\Omega)),\\
&\phi_m \in L^\infty(\Omega\times (0,T)) : |\phi_m(x,t)|\leq1-\widetilde{\delta} \ \text{a.e. in } \  \Omega\times(0,T),\\
&\mu_m \in L^\infty(0,T;H^2(\Omega))\cap W^{1,2}(0,T;L^2(\Omega)),
\end{split}
\end{equation} 
for some $\widetilde{\delta}$ depending on $\alpha$ and $k$.
We report the following estimates for the system \eqref{cCH}-\eqref{cCH-c}:
%\begin{itemize}
\smallskip

\noindent
$[1.]$ $L^2$ estimate: for any $T>0$
\begin{equation*}
\begin{split}
\sup_{t \in [0,T]} \Big( \| \phi_m(t)\|_{L^2}^2 + \alpha \| \nabla \phi_m(t)\|_{L^2}^2 \Big)&+ \int_0^T \| \Delta \phi_m(\tau)\|_{L^2}^2 \, \d \tau  \leq \| \phi_{0,k}\|_{L^2}^2 + \alpha \| \nabla \phi_{0,k}\|_{L^2}^2+ \theta_0^2 |\Omega| T;
\end{split}
\end{equation*}

\noindent
$[2.]$ Energy estimate: for any $T>0$
\begin{equation}
\label{EE-AS}
\begin{split}
\sup_{t\in [0,T]} E_{\text{free}}(\phi(t))+ \frac12 \int_0^T \| \nabla \mu_m(\tau)\|_{L^2}^2 \, \d \tau &+\alpha \int_0^T \| \partial_t \phi_m (\tau)\|_{L^2}^2 \, \d \tau \\
&\leq  E_{\text{free}}(\phi_{0,k})+ \frac12 \int_0^T \| \vv(\tau)\|_{L^2}^2 \, \d \tau.
\end{split}
\end{equation}

We now make the ansatz
$$
\uu_m(x,t)= \sum_{j=1}^m a_j^m(t) \ww_j(x)
$$
as solution to the Galerkin approximation of \eqref{AGG}$_1$ that reads as
\begin{align}
\label{w-ap1}
\begin{split}
&\begin{aligned}[t]
(\rho(\phi_m) \partial_t \uu_m, \ww_l)&+(\rho(\phi_m)(\vv\cdot \nabla)\uu_m,\ww_l)+(\nu(\phi_m) \D \uu_m, \nabla \ww_l)\\
&-\frac{\rho_1-\rho_2}{2} ( (\nabla \mu_m \cdot \nabla) \uu_m, \ww_l)= (\mu_m \nabla \phi_m, \ww_l), \quad \forall \, l=1,\dots,m,
\end{aligned}
\end{split}
\end{align}
such that $\uu_m(\cdot, 0)= \mathbb{P}_m\uu_0$.
Setting $\mathbf{A}^m(t)=(a_1^m(t), \dots, a_m^m(t))^T$, \eqref{w-ap1} is equivalent to the system of differential equations
\begin{equation}
\label{AS-ODE}
\mathbf{M}^m(t) \ddt \mathbf{A}^m+ \mathbf{L}^m(t) \mathbf{A}^m = \mathbf{G}^m(t),
\end{equation}
where the matrices $\mathbf{M}^m(t)$, $\mathbf{L}^m(t)$ and the vector $\mathbf{G}^m(t)$ are defined as
\begin{align*}
&(\mathbf{M}^m(t))_{l,j}= \int_{\Omega} \rho(\phi_m) \ww_l \cdot \ww_j \, \d x,\\
&(\mathbf{L}^m(t))_{l,j}=\int_{\Omega} 
\Big( \rho(\phi_m) (\vv \cdot \nabla) \ww_j \cdot \ww_l + \nu(\phi_m) \D \ww_j : \nabla \ww_l - \Big(\frac{\rho_1-\rho_2}{2}\Big) (\nabla \mu_m \cdot \nabla) \ww_j \cdot \ww_l \Big) \, \d x,\\
&(\mathbf{G}^m(t))_l= \int_{\Omega} \mu_m \nabla \phi_m \cdot \ww_l \, \d x,
\end{align*}
and  
$
\mathbf{A}^m(0)=((\mathbb{P}_m\uu_{0}, \ww_1), \dots, (\mathbb{P}_m \uu_{0},\ww_m))^T.
$
The regularity properties \eqref{phi-As} imply the continuity of $\phi_m \in C([0,T]; W^{1,4}(\Omega))$ and $\mu_m\in C([0,T];H^1(\Omega))$. In turn, we have $\rho(\phi_m), \nu(\phi) \in C(\overline{\Omega \times [0,T]})$. Moreover, w observe that $\vv \in C([0,T]; \LL^2_\sigma)$. Thus, we infer that $\mathbf{M}^m$ and $\mathbf{L}^m$ belong to $C([0,T];\mathbb{R}^{m \times m})$, and $\mathbf{G}^m \in C([0,T];\mathbb{R}^m)$. Since the matrix $\mathbf{M}^m(\cdot)$ is definite positive on $[0,T]$ (see \cite[Appendix A]{GT2020}), the inverse $(\mathbf{M}^m)^{-1} \in C([0,T]; \mathbb{R}^{m\times m})$. Thus, the existence and uniqueness theorem for system of linear ODEs guarantees that there exists a unique solution
 $\mathbf{A}^m \in C^1([0,T];\mathbb{R}^m)$ to \eqref{AS-ODE} on $[0,T]$. As a result, the problem \eqref{w-ap1} has a unique solution $\uu_m \in C^1([0,T];\V_m)$.
\smallskip

Next, multiplying \eqref{w-ap1} by $a_l^m$ and summing over $l$, we find
\begin{align*}
\int_{\Omega} \rho(\phi_m) \partial_t \left( \frac{|\uu_m|^2}{2} \right) \, \d x&+ \int_{\Omega} \rho(\phi_m) \vv \cdot \nabla \left(  \frac{|\uu_m|^2}{2} \right) \, \d x + \int_{\Omega} \nu(\phi_m) |\D \uu_m|^2 \, \d x\\
&- \frac{\rho_1-\rho_2}{2} \int_{\Omega} \nabla \mu_m \cdot \nabla \left(  \frac{|\uu_m|^2}{2} \right) \, \d x = \int_{\Omega} \mu_m \nabla \phi_m \cdot \uu_m \, \d x. 
\end{align*}
Integrating by parts, we obtain
\begin{align*}
\ddt \int_{\Omega} \rho(\phi_m) \frac{|\uu_m|^2}{2}\, \d x &-
\int_{\Omega} \Big( \partial_t \rho(\phi_m)+ \div \big(\rho(\phi_m) \vv \big) \Big)  \frac{|\uu_m|^2}{2}\, \d x
+ \int_{\Omega} \nu(\phi_m) |\D \uu_m|^2 \, \d x\\
&+ \frac{\rho_1-\rho_2}{2} \int_{\Omega}\Delta \mu_m  \frac{|\uu_m|^2}{2} \, \d x = \int_{\Omega} \phi_m \nabla \mu_m \cdot \uu_m \, \d x. 
\end{align*}
Recalling that $\rho'(\phi_m)= \frac{\rho_1-\rho_2}{2}$ and $\div \vv=0$, by using \eqref{cCH}$_1$, we  have
\begin{align*}
&-\int_{\Omega} \Big( \partial_t \rho(\phi_m)+ \div \big(\rho(\phi_m) \vv \big) \Big)  \frac{|\uu_m|^2}{2}\, \d x+ \frac{\rho_1-\rho_2}{2} \int_{\Omega}\Delta \mu_m  \frac{|\uu_m|^2}{2} \, \d x = 0.
\end{align*}
Thus, we infer that
\begin{equation}
\label{EE-1p}
\ddt \int_{\Omega} \rho(\phi_m) \frac{|\uu_m|^2}{2}\, \d x 
+ \int_{\Omega} \nu(\phi_m) |\D \uu_m|^2 \, \d x\\
= \int_{\Omega}  \phi_m \nabla \mu_m \cdot \uu_m \, \d x. 
\end{equation}
By using \eqref{phi-As}$_2$ and the Poincar\'{e} inequality, we get
\begin{equation*}
\begin{split}
\int_{\Omega}  \phi_m \nabla \mu_m \cdot \uu_m \, \d x &\leq 
\| \phi_m\|_{L^\infty} \| \nabla \mu_m \|_{L^2} \| \uu_m \|_{L^2}
%&\leq \frac{\sqrt{2}}{\sqrt{\lambda_1}}\| \nabla \mu_m\|_{L^2} \| \D \uu_m\|_{L^2}
\leq \frac{\nu_\ast}{2}\| \D \uu_m\|_{L^2}^2+
\frac{1}{\lambda_1 \nu_\ast} \| \nabla \mu_m\|_{L^2}^2,
\end{split}
\end{equation*}
So, we find the differential inequality
\begin{equation}
\ddt \int_{\Omega} \rho(\phi_m) \frac{|\uu_m|^2}{2}\, \d x 
+ \frac{\nu_\ast}{2}\int_{\Omega} |\D \uu_m|^2 \, \d x\\
\leq \frac{1}{\lambda_1 \nu_\ast} \| \nabla \mu_m\|_{L^2}^2.
\end{equation}
Integrating the above inequality on $[0,s]$, with $s \in [0,T]$, and using \eqref{EE-AS}, it follows that
\begin{equation}
\begin{split}
 \int_{\Omega} \frac{\rho_\ast}{2} |\uu_m(s)|^2 \, \d x 
 \leq  \int_{\Omega} \rho(\phi_{0,k}) \frac{|\mathbb{P}_m \uu_0|^2}{2}\, \d x + \frac{2}{\lambda_1 \nu_\ast}
 E_{\text{free}}(\phi_{0,k})+ \frac{1}{\lambda_1 \nu_\ast} \int_0^s \| \vv(\tau)\|_{L^2}^2 \, \d \tau,
 \end{split}
\end{equation}
which, in turn, entails that
\begin{equation}
\label{u-s-L2}
\begin{split}
\|\uu_m(s)\|_{L^2}^2  
 \leq  \frac{\rho^\ast}{\rho_\ast} \| \uu_0\|_{L^2}^2 
 + \frac{4}{\lambda_1 \rho_\ast \nu_\ast} E_{\text{free}}(\phi_{0,k})+  \frac{2}{\lambda_1\rho_\ast \nu_\ast}\int_0^s \| \vv(\tau)\|_{L^2}^2 \, \d \tau.
 \end{split}
\end{equation}
At this point, setting
$$
C_1= \frac{\rho^\ast}{\rho_\ast} \| \uu_0\|_{L^2}^2 
 + \frac{4}{\lambda_1 \rho_\ast \nu_\ast} E_{\text{free}}(\phi_{0,k}), \quad 
 C_2= \frac{2}{\lambda_1\rho_\ast \nu_\ast},
$$
and assuming
\begin{equation}
\label{ass-v}
\int_0^t \| \vv(\tau)\|_{L^2}^2 \, \d \tau \leq C_3 \mathrm{e}^{C_2 t}, \quad t \in [0,T],
\end{equation}
where $C_3=C_1 T$,
we deduce that
\begin{equation}
\label{u-L2}
\begin{split}
\int_0^t \| \uu_m (s)\|_{L^2}^2 \, \d s 
&\leq C_3 + C_2 \int_0^t \int_0^s  \| \vv(\tau)\|_{L^2}^2 \, \d \tau \, \d s \leq C_3 \mathrm{e}^{C_2 t}, \quad \forall\, t \in [0,T].
\end{split}
\end{equation}
Furthermore, thanks to \eqref{u-s-L2} and \eqref{ass-v}, we also infer that
\begin{equation}
\label{u-sup-L2}
\sup_{t \in [0,T]} \| \uu_m (t)\|_{L^2} \leq  \left( C_1+C_3 C_2 \mathrm{e}^{C_2 T} \right)^\frac12=: K_0.
\end{equation}
%where the constant $R_0$ depends on $\rho_\ast$, $\rho^\ast$, $\nu_\ast$, $\theta_0$, $\|\uu_0 \|_{L^2(\Omega)}$, $T$, $\Omega$.
Next, we control the time derivative of $\uu_m$. Multiplying \eqref{w-ap1} by $\ddt a_l^m$ and summing over $l$, we find
\begin{align*}
\rho_\ast \| \partial_t \uu_m \|_{L^2}^2 
&\leq - (\rho(\phi_m)(\vv\cdot \nabla)\uu_m,\partial_t \uu_m)-(\nu(\phi_m) \D \uu_m, \nabla \partial_t \uu_m)\\
&\quad +\frac{\rho_1-\rho_2}{2} ( (\nabla \mu_m \cdot \nabla) \uu_m, \partial_t \uu_m) +( \phi_m \nabla \mu_m, \partial_t \uu_m).
\end{align*}
By exploiting \eqref{Rev-SI}, we obtain
\begin{align*}
\rho_\ast \| \partial_t \uu_m \|_{L^2}^2 
&\leq \rho^\ast \| \vv \|_{L^2} \| \nabla\uu_m \|_{L^\infty} \| \partial_t \uu_m\|_{L^2} +
\nu^\ast \| \D \uu_m\|_{L^2} \| \nabla \partial_t \uu_m \|_{L^2}\\
&\quad +\Big| \frac{\rho_1-\rho_2}{2} \Big| \| \nabla \uu_m \|_{L^\infty}
\| \nabla \mu_m \|_{L^2} \| \partial_t \uu_m \|_{L^2} + \| \phi_m \|_{L^\infty} \| \nabla \mu_m\|_{L^2} \|\nabla \partial_t \uu_m\|_{L^2}\\
&\leq   \rho^\ast C  \| \vv\|_{L^2} \| \uu_m\|_{H^3} \| \partial_t \uu_m\|_{L^2} +  \nu^\ast C_m^2 \| \uu_m\|_{L^2} \| \partial_t \uu_m\|_{L^2} \\
&\quad + C \Big| \frac{\rho_1-\rho_2}{2} \Big| \| \uu_m\|_{H^3} \| \nabla \mu_m\|_{L^2}  \| \partial_t \uu_m \|_{L^2}+ C_m \| \nabla \mu_m\|_{L^2}  \| \partial_t \uu_m \|_{L^2}\\
&\leq   \rho^\ast C_m  \| \vv\|_{L^2} \| \uu_m\|_{L^2} \| \partial_t \uu_m\|_{L^2} +  \nu^\ast C_m^2 \| \uu_m\|_{L^2} \| \partial_t \uu_m\|_{L^2} \\
&\quad + C_m \Big| \frac{\rho_1-\rho_2}{2} \Big| \| \uu_m\|_{L^2)} \| \nabla \mu_m\|_{L^2}  \| \partial_t \uu_m \|_{L^2}
+ C_m \| \nabla \mu_m\|_{L^2}  \| \partial_t \uu_m \|_{L^2}.
\end{align*}
Then, by using \eqref{EE-AS}, \eqref{ass-v}, \eqref{u-L2} and \eqref{u-sup-L2}, we infer that
\begin{equation}
\label{utL2-AS}
\begin{split}
\int_0^T \| \partial_t \uu_m(\tau)\|_{L^2}^2 \, \d \tau
&\leq 4 \left( \frac{\rho^\ast}{\rho_\ast} C_m K_0 \right)^2 \int_0^T \| \vv(\tau)\|_{L^2}^2 \, \d \tau + 4 \left( \frac{\nu^\ast}{\rho_\ast} C_m^2  \right)^2 C_3 \mathrm{e}^{C_2 T} \\
&\quad + 4 \left( \left( \frac{C_m}{\rho_\ast} \Big| \frac{\rho_1-\rho_2}{2} \Big| K_0 \right)^2   +\frac{C_m^2}{\rho_\ast^2} \right)
\int_0^T \| \nabla \mu_m(\tau)\|_{L^2}^2 \, \d \tau  
\\
&\leq 4 \left( \left( \frac{\rho^\ast}{\rho_\ast} C_m K_0 \right)^2
+ \left( \frac{\nu_\ast}{\rho_\ast} C_m ^2 \right)^2 \right) C_3 \mathrm{e}^{C_2 T}\\
&\quad + 4 \left( \left( \frac{C_m}{\rho_\ast} \Big| \frac{\rho_1-\rho_2}{2} \Big| K_0 \right)^2 + \frac{C_m^2}{\rho_\ast^2} \right) \left( 2 E_{\text{free}}(\phi_{0,k})+ C_3 \mathrm{e}^{C_2 T}\right)=:K_1^2, 
\end{split}
\end{equation}
where $K_1$ depends only on $\rho_\ast$, $\rho^\ast$, $\nu_\ast$, $\theta_0$, $\|\uu_0 \|_{L^2}$, $E_{\text{free}}(\phi_0)$, $T$, $\Omega$, $m$.
\smallskip

Now we define the setting of the fixed point argument. We introduce the set
$$
S=\left\lbrace \uu \in W^{1,2}(0,T;\V_m): \int_0^t \| \uu(\tau)\|_{L^2}^2 \, \d \tau \leq C_3  \mathrm{e}^{C_2 t}, \, t \in [0,T], \ \| \partial_t \uu\|_{L^2(0,T;\V_m)}\leq K_1 \right\rbrace, 
$$
which is a subset of $L^2(0,T;\V_m)$. 
We define the map 
$$
\Lambda: S \rightarrow L^2(0,T;\V_m), \quad \Lambda(\vv)= \uu_m,
$$
where $\uu_m$ is the solution to the system \eqref{w-ap1}. In light of \eqref{u-L2} and \eqref{utL2-AS}, we deduce that
$\Lambda: S \rightarrow S$.
It is easily seen that $S$ is convex and closed. Furthermore, $S$ is a compact set in $L^2(0,T;\V_m)$. We are left to prove that the map $\Lambda$ is continuous. This is done by adapting the argument in \cite[Proof of Theorem 3.1]{G2021} to the viscous case. 
Let us consider a sequence $\lbrace \vv_n\rbrace\subset S$ such that $\vv_n \rightarrow \widetilde{\vv}$ in $L^2(0,T;\V_m)$. By arguing as above, there exists a sequence $\lbrace (\psi_n, \mu_n) \rbrace$ and $(\widetilde{\psi}, \widetilde{\mu})$ that solve the convective viscous Cahn-Hilliard equation \eqref{cCH}-\eqref{cCH-c}, where $\vv$ is replaced by $\vv_n$ and $\widetilde{\vv}$, respectively. 
Repeating the uniqueness argument in the proof of Theorem \ref{vCH}, we have
\begin{align*}
&\frac12 \ddt \bigg( \| \nabla A^{-1}(\psi_n -\widetilde{\psi})\|_{L^2}^2 + \alpha \| \psi_n-\widetilde{\psi}\|_{L^2}^2 \bigg) + \|\nabla (\psi_n-\widetilde{\psi})\|_{L^2}^2 \\
&\ \leq \int_{\Omega} \psi_n (\vv_n-\widetilde{\vv}) \cdot \nabla A^{-1} (\psi_n -\widetilde{\psi}) \, \d x + \int_{\Omega} (\psi_n-\widetilde{\psi}) \widetilde{\vv} \cdot \nabla A^{-1}  (\psi_n -\widetilde{\psi}) \, \d x+ \theta_0 \| \psi_n-\widetilde{\psi}\|_{L^2}^2,
\end{align*}
where the operator $A$ is the Laplace operator $-\Delta$ with homogeneous Neumann boundary conditions. Since $\widetilde{\vv}$ belong to $S$, we infer that 
\begin{align*}
\frac12 \ddt f(t) + \frac12 \|\nabla (\psi_n-\widetilde{\psi})\|_{L^2}^2  \leq C f(t) + \|\vv_n-\widetilde{\vv}\|_{L^2}^2,
\end{align*}
where $f(t)=\| \nabla A^{-1}(\psi_n(t) -\widetilde{\psi}(t))\|_{L^2}^2 + \alpha\| \psi_n(t)-\widetilde{\psi}(t)\|_{L^2}^2 $,
for some constant $C$ depending on $C_1, C_2, K_1$ and $\theta_0$. 
Observing that $\psi_n(0)-\widetilde{\psi}(0)=0$, by the Gronwall lemma we obtain
\begin{equation}
\label{wc1-AS}
\| \psi_n -\widetilde{\psi}\|_{L^\infty(0,T;L^2(\Omega))\cap L^2(0,T; H^1(\Omega))}\leq \mathrm{e}^{C T} \int_0^T \|\vv_n(\tau)-\widetilde{\vv}(\tau)\|_{L^2}^2 \, \d \tau \rightarrow 0, \quad \text{as } \ n \rightarrow \infty.
\end{equation}
On the other hand, using that $\lbrace \vv_n \rbrace $ and $\widetilde{\vv}$ belong to $S$, the continuous embedding $W^{1,2}(0,T;\V_m) \hookrightarrow Y_T$ (see Appendix A for the definition of $Y_T$) and the properties of the initial condition $\phi_{0,k}$ (cf. $\phi_{0,k}\in H^3(\Omega)$ and \eqref{psikinf}) it follows from Theorem \ref{vCH} that
\begin{align}
\label{4}
&\| \partial_t \psi_n\|_{L^\infty(0,T;H^1(\Omega))}
+ \| \partial_t \psi_n\|_{L^2(0,T;H^2(\Omega))}
\leq C, \\
\label{5}
&\| \partial_t \widetilde{\psi}\|_{L^\infty(0,T;H^1(\Omega))}
+ \| \partial_t \widetilde{\psi} \|_{L^2(0,T;H^2(\Omega))}
\leq C, 
\end{align}
for some $C$ independent of $n$.
Moreover, we also have
\begin{align}
\label{1}
&\| \mu_n\|_{L^\infty(0,T;H^2(\Omega))} + \| \psi_n\|_{L^\infty(0,T;H^3(\Omega))}\leq C,\\
\label{2}
 &\| \widetilde{\mu} \|_{L^\infty(0,T;H^2(\Omega))} + \| \widetilde{\psi}\|_{L^\infty(0,T;H^3(\Omega))}\leq C,\\
& \| \partial_t \mu_n\|_{L^2(0,T;L^2(\Omega))}\leq C, \quad \| \partial_t \widetilde{\mu}\|_{L^2(0,T;L^2(\Omega))}\leq C, 
\end{align}
and 
\begin{equation}
\label{3}
 \max_{(x,t)\in \Omega \times (0,T)} |\psi_n(x,t)|\leq 1-\delta^\ast, 
 \quad \max_{(x,t)\in \Omega \times (0,T)} |\widetilde{\psi}(x,t)|\leq 1-\delta^\ast,
\end{equation}
for some positive $C$ and $\delta^\ast \in (0,1)$, which are independent of $n$. In light of the above estimates, we first observe that $\mu_n - \widetilde{\mu} \rightarrow \mu^\ast $ in $L^\infty(0,T;L^2(\Omega))$. Our goal is to show that $\mu^\ast=0$. To this aim, we use the equation 
$$
\mu_n-\widetilde{\mu}= \varepsilon\partial_t (\psi_n-\widetilde{\psi}) 
- \Delta (\psi_n-\widetilde{\psi})+ \Psi'(\psi_n)-\Psi'(\widetilde{\psi}).
$$
By standard interpolation, we deduce from 
\eqref{wc1-AS}, \eqref{1} and \eqref{2} that 
\begin{equation}
\label{H2L-AS}
\| \psi_n -\widetilde{\psi}\|_{L^\infty(0,T;H^2(\Omega))} \rightarrow 0, \quad \text{as } \ n \rightarrow \infty.
\end{equation}
As a consequence, thanks to \eqref{3}, 
$\| \Psi'(\psi_n) - \Psi'(\widetilde{\psi})\|_{L^\infty(0,T;L^2(\Omega))} \rightarrow 0$, as $ n \rightarrow \infty$.
On the other hand, it follows from \eqref{wc1-AS}, \eqref{4} and \eqref{5} that
$\partial_t (\psi_n-\widetilde{\psi}) \rightharpoonup 0$ weakly in $ L^2(0,T;H^2(\Omega))$.
Thus, by uniqueness of the weak limit, we can conclude that
\begin{equation}
\label{MUL-AS}
\| \mu_n -\widetilde{\mu}\|_{L^\infty(0,T;L^2(\Omega))} \rightarrow 0, \quad \text{as } \ n \rightarrow \infty.
\end{equation}
We now define $\uu_n = \Lambda( \vv_n) \in S$, for any $n \in \mathbb{N}$, and $\widetilde{\uu}=\Lambda(\widetilde{\vv})\in S$. 
We consider $ \uu= \uu_n-\widetilde{\uu}$, $ \psi= \psi_n-\widetilde{\psi}$, $ \vv=\vv_n - \widetilde{\vv}$, and $ \mu= \mu_n-\widetilde{\mu}$ that solve
\begin{align}
\label{w-diff}
\begin{split}
&(\rho(\psi_n) \partial_t \uu, \ww)+ ((\rho(\psi_n)-\rho(\widetilde{\psi})) \partial_t \widetilde{\uu}, \ww)
+(\rho(\psi_n)(\vv_n\cdot \nabla)\uu_n - \rho(\widetilde{\psi})
(\widetilde{\vv}\cdot \nabla) \widetilde{\uu},\ww)\\
&\qquad +(\nu(\psi_n) \D  \uu, \nabla \ww)
+ ((\nu(\psi_n)-\nu(\widetilde{\psi})) \D \widetilde{\uu}, \nabla \ww)\\
&\qquad-\frac{\rho_1-\rho_2}{2} ( (\nabla \mu_n \cdot \nabla) \uu_n- 
(\nabla \widetilde{\mu}\cdot \nabla) \widetilde{\uu}, \ww)= (\mu_n \nabla \psi_n- \widetilde{\mu} \nabla \widetilde{\psi}, \ww), 
\end{split}
\end{align}
for all $\ww \in \V_m$, for all $t \in [0,T]$. Taking $\ww=\uu$, we obtain
\begin{align*}
&\frac12 \ddt \int_{\Omega} \rho(\psi_n) |\uu|^2 \, \d x
+\int_{\Omega} \nu(\psi_n)| \D \uu|^2 \, \d x \\
&\quad =
 \frac{\rho_1-\rho_2}{4} \int_{\Omega}  \partial_t \psi_n |\uu|^2 \, \d x-  \frac{\rho_1-\rho_2}{2} \int_{\Omega}  \psi  (\partial_t \widetilde{\uu} \cdot \uu) \, \d x \\
&\qquad - \int_{\Omega} \big(\rho(\psi_n)(\vv_n\cdot \nabla)\uu_n - \rho(\widetilde{\psi})
(\widetilde{\vv}\cdot \nabla) \widetilde{\uu} \big) \cdot \uu \, \d x- \frac{\nu_1-\nu_2}{2} \int_{\Omega}  \psi (\D \widetilde{\uu} : \D \uu) \, \d x\\
&\qquad+  \frac{\rho_1-\rho_2}{2} \int_{\Omega}
\big( (\nabla \mu_n \cdot \nabla) \uu_n - (\nabla \widetilde{\mu} \cdot \nabla) \widetilde{\uu} \big) \cdot \uu \, \d x  +\int_\Omega \big( \mu_n \nabla \psi_n- \widetilde{\mu} \nabla \widetilde{\psi} \big) \cdot \uu \, \d x.
\end{align*}
Thanks to \eqref{KORN} and \eqref{4}, we have 
\begin{align*}
 \frac{\rho_1-\rho_2}{4}  \int_{\Omega} \partial_t \psi_n |\uu|^2 \, \d x
&\leq C \| \partial_t \psi_n\|_{L^6} \| \uu\|_{L^2} \| \uu \|_{L^3}
 \leq \frac{\nu_\ast}{10} \| \D \uu\|_{L^2}^2 + C \| \uu\|_{L^2}^2,
\end{align*}
and
\begin{align*}
-\frac{\rho_1-\rho_2}{2} \int_{\Omega} \psi  (\partial_t \widetilde{\uu} \cdot \uu) \, \d x 
&\leq C \| \psi\|_{L^\infty} \| \partial_t \widetilde{\uu}\|_{L^2} \| \uu \|_{L^2}
\leq C\| \uu\|_{L^2}^2+ C  \| \partial_t \widetilde{\uu}\|_{L^2}^2 \| \psi\|_{H^2}^2.
\end{align*}
Noticing that $\vv_n$, $\widetilde{\vv}, \uu_n \in S$, by exploiting \eqref{KORN} and \eqref{Rev-SI}, we find
\begin{align*}
&-\int_{\Omega} \big(\rho(\psi_n)(\vv_n\cdot \nabla)\uu_n - \rho(\widetilde{\psi}) (\widetilde{\vv}\cdot \nabla) \widetilde{\uu} \big) \cdot \uu \, \d x\\
&\quad = - \frac{\rho_1-\rho_2}{2} \int_{\Omega} \psi ((\vv_n\cdot \nabla) \uu_n) \cdot \uu \,\ d x- \int_{\Omega} \rho(\widetilde{\psi})((\vv \cdot \nabla) \uu_n) \cdot \uu \, \d x - \int_{\Omega} \rho(\widetilde{\psi}) ((\widetilde{\vv}\cdot \nabla) \uu) \cdot \uu \, \d x \\
&\quad \leq C \| \psi\|_{L^\infty} \| \vv_n\|_{L^\infty}  \| \nabla \uu_n\|_{L^2} \| \uu\|_{L^2}  + C \| \vv\|_{L^2} \| \nabla \uu_n\|_{L^\infty} \| \uu\|_{L^2}
+ C \| \widetilde{\vv}\|_{L^\infty} \| \nabla \uu \|_{L^2} \| \uu \|_{L^2}\\
&\quad \leq C_m \| \psi\|_{H^2} \| \uu\|_{L^2}+ C_m \| \vv\|_{L^2} 
\| \uu\|_{L^2}+ C \| \nabla \uu\|_{L^2} \| \uu\|_{L^2}\\
&\quad \leq   \frac{\nu_\ast}{10} \| \D \uu\|_{L^2}^2 + C_m \| \uu\|_{L^2}^2
+ C_m \| \psi\|_{H^2}^2 + C_m \| \vv\|_{L^2}^2.
\end{align*}
In addition, we deduce that
\begin{align*}
- \frac{\nu_1-\nu_2}{2} \int_{\Omega} \psi (\D \widetilde{\uu} : \D \uu) \, \d x & \leq C \| \psi\|_{L^\infty} \| \D \widetilde{\uu}\|_{L^2} \| \D \uu \|_{L^2}
\leq  \frac{\nu_\ast}{10} \| \D \uu\|_{L^2}^2 + C_m \| \psi\|_{H^2}^2,
\end{align*}
and
\begin{align*}
\frac{\rho_1-\rho_2}{2}& \int_{\Omega}
\left( (\nabla \mu_n \cdot \nabla) \uu_n - (\nabla \widetilde{\mu} \cdot \nabla) \widetilde{\uu} \right) \cdot \uu \, \d x \\
&=- \frac{\rho_1-\rho_2}{2} \int_{\Omega} (\mu_n \Delta \uu_n-\widetilde{\mu} \Delta \widetilde{\uu} ) \cdot \uu \, \d x
-\frac{\rho_1-\rho_2}{2} \int_{\Omega} (\mu_n \nabla \uu_n-\widetilde{\mu} \nabla \widetilde{\uu}): \nabla \uu \, \d x\\
&= -\frac{\rho_1-\rho_2}{2} \int_{\Omega} (\mu \Delta \uu_n+\widetilde{\mu} \Delta \uu ) \cdot \uu \, \d x
-\frac{\rho_1-\rho_2}{2} \int_{\Omega} (\mu \nabla \uu_n+\widetilde{\mu} \nabla \uu): \nabla \uu \, \d x\\
&\leq C \| \mu\|_{L^2} \| \Delta \uu_n\|_{L^2} \| \uu\|_{L^\infty}+
C \| \widetilde{\mu}\|_{L^6} \| \Delta \uu\|_{L^2} \| \uu\|_{L^3} \\
&\quad + C \| \mu\|_{L^2} \| \nabla \uu_n\|_{L^6} \| \nabla \uu\|_{L^3} + C \| \widetilde{\mu}\|_{L^6} \| \nabla \uu\|_{L^6} \| \nabla \uu\|_{L^3}\\
& \leq C_m \| \mu\|_{L^2} \| \nabla \uu\|_{L^2} + C_m \| \nabla \uu \|_{L^2} \| \uu\|_{L^2}\\
&\leq  \frac{\nu_\ast}{10} \| \D \uu\|_{L^2}^2 + C_m \| \mu\|_{L^2}^2+ C_m \| \uu \|_{L^2}^2.
\end{align*}
Finally, by \eqref{1}-\eqref{2}, we have
\begin{align*}
\int_\Omega \big( \mu_n \nabla \psi_n- \widetilde{\mu} \nabla \widetilde{\psi} \big) \cdot \uu \, \d x
&\leq \left( \| \mu \|_{L^2} \| \nabla \psi_n\|_{L^6}
+ \| \widetilde{\mu}\|_{L^2}  \| \nabla \psi\|_{L^6}  \right) \| \uu\|_{L^3}\\
&\leq C \left( \| \mu\|_{L^2}+ \| \psi\|_{H^2} \right) \| \nabla \uu\|_{L^2}\\
&\leq  \frac{\nu_\ast}{10} \| \D \uu\|_{L^2}^2 
+ C \| \mu\|_{L^2}^2 + C \| \psi \|_{H^2}^2.
\end{align*}
Combining the above inequalities, we are led to the differential inequality
\begin{align*}
& \ddt \int_{\Omega} \rho(\psi_n) |\uu|^2 \, \d x
\leq h_1(t) \int_{\Omega} \rho(\psi_n) |\uu|^2 \, \d x + h_2(t),
\end{align*}
where
$$
h_1(t)= C_m \big(1+ \| \partial_t \psi_n(t)\|_{H^1}^2\big)
$$
and 
$$
h_2(t)= C_m\big( \| \partial_t \widetilde{\uu}(t)\|_{L^2}^2 \| \psi(t)\|_{H^2}^2+ \| \psi(t)\|_{H^2}^2 + \| \vv (t)\|_{L^2}^2+ \| \mu(t)\|_{L^2}^2\big).
$$
Thus, the Gronwall lemma entails
$$
\sup_{t \in [0,T]} \| \uu(t)\|_{L^2}^2 \leq \frac{1}{\rho_\ast} \mathrm{e}^{\int_0^T h_1(\tau) \d \tau} \int_0^T h_2(\tau) \, \d \tau.
$$
On account of \eqref{4}, \eqref{H2L-AS}, \eqref{MUL-AS}, and the convergence $\vv_n \rightarrow \widetilde{\vv}$ in $L^2(0,T;\V_m)$, we deduce that $\uu_n \rightarrow \widetilde{\uu}$ in $L^\infty(0,T;\V_m)$, implying that the map $\Lambda$ is continuous.
Finally, we are in the position to apply the Schauder fixed point theorem and conclude that the map $\Lambda$ has a fixed point in $S$, which gives the existence of the approximate solution $(\uu_m, \phi_m)$ on $[0,T]$  satisfying \eqref{reg-AS}-\eqref{bc-AS} for any $m \in \mathbb{N}$.

\subsection{Uniform estimates independent of the approximation parameters}
First, integrating \eqref{CH-w}$_1$ over $\Omega$
\begin{equation}
\label{cons-mass}
\int_{\Omega} \phi_m(t) \, \d x= \int_{\Omega} \phi_{0,k} \, \d x, \quad \forall \, t \in [0,T].
\end{equation}
Owing to \eqref{psiH1}, for $k > \overline{k}$, $|\overline{\phi_m}(t)|\leq \widetilde{m}<1$ for all $t\in [0,T]$.
Taking $\ww=\uu_m$ in \eqref{NS-w} and integrating by parts, we have (cf. \eqref{EE-1p})
\begin{equation}
\label{NS-u}
\ddt \int_{\Omega} \frac12 \rho(\phi_m) |\uu_m|^2 \, \d x 
 + \int_{\Omega} \nu(\phi_m)|\D \uu_m|^2 \, \d x= 
 \int_{\Omega} \mu_m \nabla \phi_m \cdot \uu_m \, \d x. 
\end{equation}
Multiplying \eqref{cCH} by $\mu_m$, integrating over $\Omega$ and exploiting the definition of $\mu_m$, we find
\begin{equation}
\label{CH-mu}
\ddt \left( \int_{\Omega} \frac12 |\nabla \phi_m|^2 + \Psi(\phi_m) \, \d x\right)
+ \int_{\Omega} |\nabla \mu_m|^2 +\alpha  |\partial_t \phi_m|^2 \, \d x + \int_{\Omega} \uu_m \cdot \nabla \phi_m \mu_m \, \d x=0.
\end{equation}
By summing \eqref{NS-u} and \eqref{CH-mu}, we reach
\begin{equation}
\label{EE-m}
\ddt E(\uu_m, \phi_m)
 + \int_{\Omega} \nu(\phi_m)| \D \uu_m|^2 \, \d x
 + \int_{\Omega} |\nabla \mu_m|^2 \, \d x=0.
\end{equation}
An integration in time on $[0,t]$, with $0<t\leq T$, yields
$$
E(\uu_m(t), \phi_m(t))+ \int_0^t \int_{\Omega} \nu(\phi_m)| \D \uu_m|^2 \, \d x + \int_0^t \int_{\Omega} |\nabla \mu_m|^2 \, \d x= E(\mathbb{P}_m\uu_0,\phi_{0,k}), \quad \forall \, t \in [0,T].
$$
Thanks to \eqref{psiH1} and \eqref{psikinf}, we observe that
\begin{align*}
E(\mathbb{P}_m\uu_0,\phi_{0,k}) 
&\leq \frac{\rho^\ast}{2} \| \uu_0\|_{L^2}^2+ \frac12 \| \phi_0\|_{H^1}^2+ \theta_0 \left( 1+ |\Omega| \max_{s\in [-1,1]} \left| \Psi(s) \right| \right).
\end{align*}
Since $\phi_m \in L^\infty(\Omega\times (0,T)):
|\phi_m(x,t)|<1$ almost everywhere in $\Omega\times(0,T)$, we obtain
\begin{align}
\label{uL2H1}
&\| \uu_m\|_{L^\infty(0,T;\LL^2_\sigma)} + \| \uu_m\|_{L^2(0,T;\H^1_\sigma)}\leq C,\\
\label{phiH1}
&\| \phi_m \|_{L^\infty(0,T;H^1(\Omega))}\leq C, \\
\label{nmuL2}
& \| \nabla \mu_m\|_{L^2(0,T;\L2)}\leq C,\\
\label{alpha-phit}
& \sqrt{\alpha}\| \partial_t \phi_m \|_{L^2(0,T;\L2)}\leq C,
\end{align}
where the constant $C$ depends on $\| \uu_0\|_{L^2}$ and $\| \phi_0\|_{H^1}$, but is independent of $m$, $\alpha$ and $k$. 
Multiplying \eqref{cCH} by $-\Delta \phi_m$, integrating over $\Omega$ and using \eqref{phi-As}, we get
$$
\| \Delta \phi_m \|_{L^2}^2 + \int_{\Omega} F''(\phi_m) |\nabla \phi_m|^2 \, \d x= \alpha \int_{\Omega} \partial_t \phi_m \Delta \phi_m \, \d x+ \int_{\Omega} \nabla \mu_m \cdot \nabla \phi_m \, \d x+ \theta_0 \|\nabla \phi_m\|_{L^2}^2.
$$
Since $F''(s)>0$ for $s\in (-1,1)$, by using \eqref{phiH1}, we have
\begin{equation}
\label{phiH2e}
\| \Delta \phi_m \|_{L^2}^2  \leq C \left( 1+ \| \nabla \mu_m\|_{L^2}^2+ \alpha^2 \| \partial_t \phi_m\|_{L^2}^2\right),
\end{equation}
for some $C$ independent of $m$. Then, it follows from \eqref{nmuL2} and \eqref{alpha-phit} that
\begin{equation}
\label{phiH2}
\| \phi_m\|_{L^2(0,T;H^2(\Omega))}\leq C.
\end{equation}
We now recall the well-known inequality (see \cite{MZ})
\begin{equation}
\label{F'-L1}
\int_{\Omega} |F'(\phi_m)| \, \d x \leq C \int_{\Omega} F'(\phi_m) (\phi_m -\overline{\phi_{0,k}}) \, \d x+ C,
\end{equation}
where the constant $C$ depends only on $\overline{\phi_{0,k}}$, thereby it is independent of $k$ (for $k$ large).
Then, multiplying \eqref{CH-w}$_2$ by $\phi_m - \overline{\phi_{0,k}}$ (cf. \eqref{cons-mass}), we find
\begin{align*}
\int_{\Omega} |\nabla \phi_m|^2 \, \d x&+ 
\int_{\Omega} F'(\phi_m) (\phi_m -\overline{\phi_{0,k}}) \, \d x \\
&= - \alpha \int_{\Omega} \partial_t \phi_m \left(\phi_m-\overline{\phi_{0,k}}\right)\, \d x+ \int_{\Omega} (\mu-\overline{\mu}) \phi_m \, \d x + \theta_0 \int_{\Omega} \phi_m (\phi_m -\overline{\phi_{0,k}}) \, \d x.
\end{align*}
By the Poincar\'{e} inequality and \eqref{phiH1}, we obtain
\begin{equation}
\label{F'-L1e}
\left| \int_{\Omega} F'(\phi_m) (\phi_m -\overline{\phi_{0,k}}) \, \d x \right| \leq 
C \left(1+ \| \nabla \mu_m\|_{L^2} + \alpha \| \partial_t \phi_m\|_{L^2}\right).
\end{equation}
Since $\overline{\mu_m}= \overline{F'(\phi_m)}- \theta_0 \overline{\phi_{0,k}}$, we infer from \eqref{F'-L1} and \eqref{F'-L1e} that
$$
|\overline{\mu_m}|\leq C\left(1+ \| \nabla \mu_m\|_{L^2} + \alpha \| \partial_t \phi_m\|_{L^2}\right).
$$ 
Thanks to \eqref{normH1-2}, we have
\begin{equation}
\label{mu-H1e}
\| \mu_m\|_{H^1}\leq 
C \left( 1+ \| \nabla \mu_m\|_{L^2}+ \alpha \| \partial_t \phi_m\|_{L^2} \right).
\end{equation}
As a direct consequence, we deduce that
 \begin{equation}
 \label{mu-H1}
 \| \mu_m\|_{L^2(0,T;H^1(\Omega))}\leq C,
 \end{equation}
for some constant $C$ independent of $m$, $\alpha$ and $k$. 
In addition, using the boundary conditions \eqref{bc-AS} and \eqref{uL2H1}, we find
\begin{equation}
\label{phit}
\| \partial_t \phi_m\|_{(H^1)'}\leq C \left (1+ \|\nabla \mu_m \|_{L^2} \right),
\end{equation}
which, in turn, implies that
$$
\| \partial_t \phi_m\|_{L^2(0,T;(H^1(\Omega))')}\leq C.
$$

Next, taking $\ww=\partial_t \uu_m$ in \eqref{NS-w}, we find
\begin{equation}
\label{High1}
\begin{split}
&\frac12 \ddt \int_{\Omega} \nu(\phi_m) | \D \uu_m|^2 \, \d x+ \int_{\Omega} \rho(\phi_m) |\partial_t \uu_m|^2 \, \d x\\
&\quad =- \int_{\Omega} \rho(\phi_m)  ((\uu_m \cdot \nabla) \uu_m ) \cdot \partial_t \uu_m \, \d x +\frac{\nu_1-\nu_2}{2} \int_{\Omega}  \partial_t \phi _m |\D \uu_m|^2 \, \d x\\
&\qquad + \frac{\rho_1-\rho_2}{2}\int_{\Omega} ((\nabla \mu_m \cdot \nabla) \uu_m ) \cdot \partial_t \uu_m \, \d x + \int_{\Omega}  \mu_m \nabla \phi_m \cdot \partial_t \uu_m \, \d x.
\end{split}
\end{equation}
Thanks to the regularity of $\mu$ (cf. \eqref{phi-As}), we multiply \eqref{CH-w}$_1$ by $\partial_t \mu_m$ and integrate over $\Omega$
$$
\frac12 \ddt \int_{\Omega} |\nabla \mu_m|^2 \, \d x + ( \partial_t \mu_m, \partial_t \phi_m ) + (\partial_t \mu_m, \uu_m \cdot \nabla \phi_m )=0.
$$
Direct computations give that
\begin{align*}
( \partial_t \mu_m, \partial_t \phi_m )= \alpha (\partial_{tt} \phi_m, \partial_t \phi_m)+ \| \nabla \partial_t \phi_m\|_{L^2}^2+
\int_{\Omega} F''(\phi_m) |\partial_t \phi_m|^2 \,\d x - \theta_0 \| \partial_t \phi_m\|_{L^2}^2
\end{align*}
and
$$
( \partial_t \mu_m, \uu_m \cdot \nabla \phi_m )=
\ddt \left( \int_{\Omega} \mu_m \uu_m \cdot \nabla \phi_m\, \d x \right)
- \int_{\Omega}  \mu_m \partial_t \uu_m \cdot \nabla \phi_m \, \d x
-\int_{\Omega} \mu_m \uu_m \cdot \nabla \partial_t \phi_m \, \d x.
$$
As a result, we  find
\begin{equation}
\label{High2}
\begin{split}
&\ddt \bigg( \int_{\Omega} \frac12 |\nabla \mu_m|^2 \, \d x + \int_{\Omega}
\frac{\alpha}{2} |\partial_t \phi_m|^2 \, \d x
+ \int_{\Omega} \mu_m \uu_m \cdot \nabla \phi_m \, \d x \bigg) + \| \nabla \partial_t \phi_m\|_{L^2}^2\\
&\quad  \leq \theta_0 \| \partial_t \phi_m \|_{L^2}^2 + \int_{\Omega}  \mu_m \partial_t \uu_m \cdot \nabla \phi_m \, \d x
+ \int_{\Omega} \mu_m \uu_m \cdot \nabla \partial_t \phi_m \, \d x.
\end{split}
\end{equation}
By summing \eqref{High1} and \eqref{High2}, we arrive at
\begin{equation}
\label{High3}
\begin{split}
\ddt H_m &+ \rho_\ast \|\partial_t \uu_m\|_{L^2}^2
+ \| \nabla \partial_t \phi_m\|_{L^2}^2\\
&\quad \leq - \int_{\Omega} \rho(\phi_m) ( (\uu_m \cdot \nabla) \uu_m ) \cdot \partial_t \uu_m \, \d x +  \frac{\nu_1-\nu_2}{2} \int_{\Omega}\partial_t \phi _m |\D \uu_m|^2 \, \d x\\
&\qquad + \frac{\rho_1-\rho_2}{2} \int_{\Omega} ((\nabla \mu_m \cdot \nabla) \uu_m ) \cdot \partial_t \uu_m \, \d x  + 2 \int_{\Omega} \mu_m \nabla \phi_m \cdot \partial_t \uu_m \, \d x\\
&\qquad  +\theta_0 \| \partial_t \phi_m \|_{L^2}^2 
+ \int_{\Omega} \mu_m \uu_m \cdot \nabla \partial_t \phi_m \, \d x\\
&\quad = \sum_{k=1}^6 R_i,
\end{split}
\end{equation}
where
$$
H_m(t)= \frac12 \int_{\Omega} \nu(\phi_m) |\D \uu_m|^2 \, \d x +
\frac12 \int_{\Omega} |\nabla \mu_m|^2 \, d x 
+ \frac{\alpha}{2} \int_{\Omega} |\partial_t \phi_m|^2 \, \d x 
+ \int_{\Omega} \mu_m \uu_m \cdot \nabla \phi_m \, \d x.
$$
By exploiting \eqref{LADY3}, \eqref{KORN}, \eqref{uL2H1}, \eqref{phiH1}, and \eqref{mu-H1e}, we observe that
\begin{align*}
\left| \int_{\Omega} \mu_m \uu_m \cdot \nabla \phi_m \, \d x \right|
&\leq \| \mu_m\|_{L^6} \| \uu_m\|_{L^3} \| \nabla \phi_m \|_{L^2}\\
&\leq C \left( 1+ \| \nabla \mu_m\|_{L^2} + \alpha \| \partial_t \phi_m\|_{L^2} \right) \|\nabla \uu_m \|_{L^2}^\frac12\\
&\leq \frac{1}{4} \int_{\Omega}\nu(\phi_m) | \D \uu_m|^2 \, \d x + \frac14 \|\nabla \mu_m\|_{L^2}^2 + \frac{\alpha}{4} \| \partial_t \phi_m\|_{L^2}^2+C_0,
\end{align*}
for some $C_0$ independent of $m$, $\alpha$ and $k$. Thus, it follows that
\begin{equation}
\label{H-b}
H_m \geq  \frac{1}{4} \int_{\Omega}\nu(\phi_m) | \D \uu_m|^2 \, \d x + \frac14 \|\nabla \mu_m\|_{L^2}^2 + \frac{\alpha}{4} \| \partial_t \phi_m\|_{L^2}^2  -C_0.
\end{equation}
Similarly, it is easily seen that
\begin{equation}
\label{H-ab}
H_m \leq  \int_{\Omega}\nu(\phi_m) | \D \uu_m|^2 \, \d x + \|\nabla \mu_m\|_{L^2}^2 + \alpha\| \partial_t \phi_m\|_{L^2}^2  +\widetilde{C}_0, 
\end{equation}
for some $\widetilde{C}_0$ independent of $m$, $\alpha$ and $k$.
Before proceeding with the estimate of the terms $R_i$, $i=1, \dots, 7$, we need to control the norms $\| \A \uu_m\|_{L^2}$ and $\| \mu_m\|_{H^3}$. To this aim, taking $\ww=\A \uu_m$ in \eqref{w-ap1}, we have
\begin{equation}
\label{A-H2}
\begin{split}
-\frac12 (\nu(\phi_m) \Delta \uu_m,  \A \uu_m)&=
-(\rho(\phi_m) \partial_t \uu_m, \A\uu_m)
-(\rho(\phi_m)(\uu_m\cdot \nabla)\uu_m,\A \uu_m)\\
&\quad+\frac{\rho_1-\rho_2}{2} ( (\nabla \mu_m \cdot\nabla) \uu_m, \A \uu_m)
+( \mu_m \nabla \phi_m,  \A \uu_m)\\
&\quad +\frac{\nu_1-\nu_2}{2}( \D \uu_m \nabla \phi_m, \A \uu_m).
\end{split}
\end{equation}
By arguing as in \cite{GMT2019} (see also \cite{G2021}), there exists $\pi_m \in C([0,T];H^1(\Omega))$ such that $-\Delta \uu_m + \nabla \pi_m= \A \uu_m$ almost everywhere in $\Omega \times (0,T)$ and satisfies
\begin{equation}
\label{p-est}
\| \pi_m\|_{L^2}\leq C\| \nabla \uu_m\|_{L^2}^\frac12 \| \A \uu_m\|_{L^2}^\frac12,\quad \| \pi_m\|_{H^1}\leq C \| \A \uu_m\|_{L^2},
\end{equation}
where $C$ is independent of $m$, $\alpha$ and $k$. Therefore, we obtain
\begin{equation}
\label{High4}
\begin{split}
\frac12 (\nu(\phi_m) \A \uu_m,  \A \uu_m)&=
-(\rho(\phi_m) \partial_t \uu_m, \A\uu_m)
-(\rho(\phi_m)(\uu_m\cdot \nabla)\uu_m,\A \uu_m)\\
&\quad+\frac{\rho_1-\rho_2}{2} ( (\nabla \mu_m \cdot \nabla) \uu_m, \A \uu_m)+(\mu_m \nabla \phi_m,  \A \uu_m)\\
&\quad +\frac{\nu_1-\nu_2}{2}( \D \uu_m \nabla \phi_m, \A \uu_m)
-\frac{\nu_1-\nu_2}{4} (\pi_m \nabla \phi_m, \A \uu_m)\\
&= \sum_{i=7}^{12} R_i.
\end{split}
\end{equation}
On the other hand, taking the gradient of \eqref{CH-w}$_1$, multiplying it by $\nabla \Delta \mu$ and integrating over $\Omega$, we find
\begin{equation}
\label{mu-H3}
\begin{split}
\|\nabla \Delta \mu_m\|_{L^2}^2
&= (\nabla \partial_t \phi_m, \nabla \Delta \mu_m)+
( \nabla (\uu_m \cdot \nabla \phi_m), \nabla \Delta \mu_m).
\end{split}
\end{equation}
Then, in light of \eqref{CH-w}$_1$ and \eqref{bc-AS}$_1$, it follows that
$$
\| \mu_m\|_{H^3}^2\leq C \left( \| \mu_m\|_{H^1}^2+ \| \nabla \Delta \mu_m\|_{L^2}^2 \right),
$$
which, in turn, by \eqref{H-b} gives that
\begin{equation}
\label{mu-H3-f}
\begin{split}
\| \mu_m\|_{H^3}^2&\leq C \Big( 1+ \| \nabla \mu_m\|_{L^2}^2+ \alpha^2 \| \partial_t \phi_m \|_{L^2}^2  + (\nabla \partial_t \phi_m, \nabla \Delta \mu_m)+
( \nabla (\uu_m \cdot \nabla \phi_m), \nabla \Delta \mu_m) \Big)\\
&= C \left(1+ C_0 +H_m \right) +\sum_{i=13}^{14} R_i,
\end{split}
\end{equation}
where $C$ is independent of $m$, $\alpha$ and $k$.
Now, multiplying \eqref{High4} and \eqref{mu-H3-f} by two positive constants $\varpi_1$ and $\varpi_2$ (which will be chosen later on), respectively, and summing them to \eqref{High3}, we obtain
\begin{equation}
\label{High5}
\begin{split}
\ddt H_m &+ \rho_\ast \|\partial_t \uu_m\|_{L^2}^2
+ \| \nabla \partial_t \phi_m\|_{L^2}^2
+ \frac{\nu_\ast \varpi_1}{2} \| \A \uu_m\|_{L^2}^2
+ \varpi_2 \| \mu_m\|_{H^3}^2
\\
&\quad \leq C(1+\varpi_2) \left(1+ C_0+H_m \right)+ \sum_{i=1}^{6} R_i+ \varpi_1 \sum_{i=7}^{12} R_i
+ \varpi_2 \sum_{i=13}^{14} R_i.
\end{split}
\end{equation}
Let us proceed with the estimate of the terms $R_i$, $i=1,\dots, 14$. In the sequel the generic constant $C$ may depend on $\varpi_1$ and $\varpi_2$.
Exploiting \eqref{LADY3}, \eqref{KORN}, \eqref{uL2H1} and \eqref{H-b}, we  have
\begin{equation*}
\label{I1}
\begin{split}
\left| - \int_{\Omega} \rho(\phi_m) ( (\uu_m \cdot \nabla) \uu_m ) \cdot \partial_t \uu_m \, \d x \right| 
%\Big| - \int_{\Omega} \rho(\phi_m) \big( (\uu_m \cdot \nabla) \uu_m \big) \cdot \partial_t \uu_m \, \d x \Big|
&\leq \rho^\ast \| \uu_m \|_{L^6} \| \nabla \uu_m\|_{L^3} \| \partial_t \uu_m\|_{L^2}\\
& \leq \frac{\rho_\ast}{8} \| \partial_t \uu_m\|_{L^2}^2+
C\| \nabla \uu_m \|_{L^2}^3 \| \A \uu_m\|_{L^2}\\
&\leq  \frac{\rho_\ast}{8} \| \partial_t \uu_m\|_{L^2}^2
+\frac{\nu_\ast \varpi_1 }{32} \| \A \uu_m\|_{L^2}^2+
C\| \D \uu_m \|_{L^2}^6\\
&\leq  \frac{\rho_\ast}{8} \| \partial_t \uu_m\|_{L^2}^2
+\frac{\nu_\ast \varpi_1 }{32} \| \A \uu_m\|_{L^2}^2+
C \left( C_0+ H_m \right)^3.
\end{split}
\end{equation*}
By Sobolev embedding, \eqref{LADY3} and \eqref{H-b}, we obtain
\begin{equation*}
\label{I2}
\begin{split}
\left| \frac{\nu_1-\nu_2}{2} \int_{\Omega}\partial_t \phi _m |\D \uu_m|^2 \, \d x \right| 
&\leq C \| \partial_t \phi_m\|_{L^6} \| \D \uu_m  \|_{L^3} \| \D \uu_m  \|_{L^2} \\
&\leq \frac{1}{8} \| \nabla \partial_t \phi_m\|_{L^2}^2 
+C  \| \A \uu_m \|_{L^2} \| \D \uu_m \|_{L^2}^3\\
&\leq \frac{1}{8} \| \nabla \partial_t \phi_m\|_{L^2}^2 + \frac{\nu_\ast \varpi_1}{32} \| \A \uu_m\|_{L^2}^2 + 
C \| \D \uu_m \|_{L^2}^3\\
&\leq \frac{1}{8} \| \nabla \partial_t \phi_m\|_{L^2}^2 + \frac{\nu_\ast \varpi_1}{32} \| \A \uu_m\|_{L^2}^2 + C \left( C_0+ H_m \right)^3.
\end{split}
\end{equation*}
By Sobolev interpolation, \eqref{Agmon3d} and \eqref{mu-H1e}, we get
\begin{equation*}
\label{I3}
\begin{split}
\left| \frac{\rho_1-\rho_2}{2} \int_{\Omega} ((\nabla \mu_m \cdot \nabla) \uu_m ) \cdot \partial_t \uu_m \, \d x \right| 
&\leq C  \| \nabla \mu_m\|_{L^\infty} 
\| \nabla \uu_m \|_{L^2} \| \partial_t \uu_m \|_{L^2}\\
& \leq C  \| \nabla \mu_m\|_{H^1}^\frac12 \| \mu_m \|_{H^3}^\frac12 \| \nabla \uu_m \|_{L^2} \| \partial_t \uu_m\|_{L^2}\\
&\leq \frac{\rho_\ast}{8} \| \partial_t \uu_m\|_{L^2}^2 
+C \| \nabla \mu_m\|_{L^2}^\frac12 \| \mu_m\|_{H^3}^\frac32 \| \D \uu_m \|_{L^2}^2 \\
&\leq \frac{\rho_\ast}{8} \| \partial_t \uu_m\|_{L^2}^2 
+ \frac{\varpi_2}{6} \| \mu_m\|_{H^3}^2
+ C \| \nabla \mu_m\|_{L^2}^2 \| \D \uu_m\|_{L^2}^8\\
&\leq \frac{\rho_\ast}{8} \| \partial_t \uu_m\|_{L^2}^2 
+ \frac{\varpi_2}{6} \| \mu_m\|_{H^3}^2
+ C \left( C_0+ H_m \right)^5.
\end{split}
\end{equation*}
Exploiting \eqref{phiH2e}, \eqref{mu-H1e},\eqref{phit} and \eqref{H-b}, we find
\begin{equation*}
\label{I4}
\begin{split}
\left| 2 \int_{\Omega} \mu_m \nabla \phi_m \cdot \partial_t \uu_m \, \d x \right| 
&\leq 2 \| \mu_m\|_{L^6} \| \nabla \phi_m \|_{L^3}
\| \partial_t \uu_m\|_{L^2}\\
&\leq \frac{\rho_\ast}{8} \| \partial_t \uu_m\|_{L^2}^2 +
C \| \phi_m\|_{H^2}^2 \| \mu_m\|_{H^1}^2 \\
&\leq \frac{\rho_\ast}{8} \| \partial_t \uu_m\|_{L^2}^2 +
C  \left( 1+ \| \nabla \mu_m\|_{L^2}^2+ \alpha^2 \| \partial_t \phi_m\|_{L^2}^2\right)^2\\
&\leq \frac{\rho_\ast}{8} \| \partial_t \uu_m\|_{L^2}^2 +
C  \left( 1 + C_0 + H_m \right)^2,
\end{split}
\end{equation*}
\begin{equation*}
\label{I5}
\begin{split}
\theta_0 \| \partial_t \phi_m \|_{L^2}^2
&\leq C \| \partial_t \phi_m\|_{(H^1)'} \| \nabla \partial_t \phi_m \|_{L^2}\\
&\leq \frac18 \|\nabla \partial_t \phi_m \|_{L^2}^2
 +C   \left( 1 + C_0 + H_m \right),
\end{split}
\end{equation*}
and
\begin{equation*}
\label{I6}
\begin{split}
\left|  \int_{\Omega} \mu_m \uu_m \cdot \nabla \partial_t \phi_m \, \d x\right| 
&\leq \| \mu_m \|_{L^6} \|  \uu_m\|_{L^3} \| \nabla \partial_t \phi_m\|_{L^2}\\
&\leq \frac18 \| \nabla \partial_t \phi_m\|_{L^2}^2 +C \| \D \uu_m\|_{L^2}^2 
\left(1+\|\nabla \mu_m\|_{L^2}^2+ \alpha^2 \| \partial_t \phi_m\|_{L^2}^2\right)\\
&\leq \frac18 \|\nabla \partial_t \phi_m \|_{L^2}^2
 +C   \left( 1 + C_0 + H_m \right)^2.
\end{split}
\end{equation*}
By Young's inequality, we have
\begin{equation*}
\label{I7}
\begin{split}
\left|-\int_{\Omega} \rho(\phi_m) \partial_t \uu_m \cdot \A\uu_m \, \d x \right|
&\leq \varpi_1 \rho^\ast \| \partial_t \uu_m\|_{L^2} \| \A \uu_m\|_{L^2}\\
&\leq \frac{\rho_\ast}{8 \varpi_1} \| \partial_t \uu_m\|_{L^2}^2
+\frac{2\left(\rho^\ast \right)^2 \varpi_1}{\rho_\ast} \| \A \uu_m\|_{L^2}^2.
\end{split}
\end{equation*}
By using \eqref{LADY3}, \eqref{Agmon3d}, \eqref{KORN} and \eqref{H-b}, we find
\begin{equation*}
\label{I8}
\begin{split}
 \left| - \int_{\Omega} \rho(\phi_m)(\uu_m\cdot \nabla)\uu_m \cdot \A \uu_m \, \d x \right|
&\leq  \rho^\ast \| \uu_m\|_{L^6} \| \nabla \uu_m\|_{L^3}
\| \A \uu_m\|_{L^2}\\
&\leq C \| \D \uu_m\|_{L^2}^\frac32 \| \A \uu_m \|_{L^2}^\frac32\\
&\leq \frac{\nu_\ast}{32} \| \A \uu_m\|_{L^2}^2 +
C \| \D \uu_m\|_{L^2}^6\\
&\leq \frac{\nu_\ast}{32} \| \A \uu_m\|_{L^2}^2 +
C\left( C_0+H_m \right)^3,
\end{split}
\end{equation*}
and
\begin{equation*}
\label{I9}
\begin{split}
\left| \frac{\rho_1-\rho_2}{2} \int_{\Omega} (\nabla \mu_m \cdot \nabla) \uu_m \cdot \A \uu_m \, \d x \right|
&\leq  C  \|\nabla \mu_m\|_{L^\infty} 
\| \nabla \uu_m\|_{L^2} \| \A \uu_m \|_{L^2}\\
&\leq C \| \nabla \mu_m \|_{H^1}^\frac12 \| \mu_m\|_{H^3}^\frac12
 \| \nabla \uu_m\|_{L^2} \| \A \uu_m\|_{L^2} \\
&\leq \frac{\nu_\ast }{32} \| \A \uu_m\|_{L^2}^2  
+ C \| \nabla \mu_m\|_{L^2}^\frac12 
\| \mu_m\|_{H^3}^\frac32  \| \D \uu_m\|_{L^2}^2\\
&\leq  \frac{\nu_\ast }{32} \| \A \uu_m\|_{L^2}^2  
+\frac{\varpi_2}{6 \varpi_1} \| \mu_m\|_{H^3}^2
+ C\| \nabla \mu_m\|_{L^2}^2  \| \D \uu_m\|_{L^2}^8\\
&\leq  \frac{\nu_\ast }{32} \| \A \uu_m\|_{L^2}^2  
+\frac{\varpi_2}{6 \varpi_1} \| \mu_m\|_{H^3}^2
+ C \left( C_0 +H_m \right)^5.
\end{split}
\end{equation*}
In light of \eqref{phiH2e} and \eqref{mu-H1e}, we have
\begin{equation*}
\label{I10}
\begin{split}
 \left| \int_{\Omega} \mu_m \nabla \phi_m \cdot  \A \uu_m \, \d x\right|
&\leq \| \mu_m\|_{L^6} \| \nabla \phi_m \|_{L^3} \| \A \uu_m\|_{L^2}\\
&\leq \frac{\nu_\ast }{32} \| \A \uu_m\|_{L^2}^2  +
C \| \mu_m\|_{H^1}^2 \| \phi_m\|_{H^2}^2\\
&\leq \frac{\nu_\ast }{32} \| \A \uu_m\|_{L^2}^2  +
 C  \left( 1+ \| \nabla \mu_m\|_{L^2}^2+ \alpha^2 \| \partial_t \phi_m\|_{L^2}^2\right)^2 \\
 &\leq \frac{\nu_\ast}{32} \| \A \uu_m\|_{L^2}^2  +
 C  \left( 1+ C_0 +H_m \right)^2,
 \end{split}
\end{equation*}
and
\begin{equation*}
\label{I11}
\begin{split}
 \left| \frac{\nu_1-\nu_2}{2} \int_{\Omega} \D \uu_m \nabla \phi_m \cdot \A \uu_m \, \d x
\right|
&\leq C \| \D \uu_m\|_{L^3} \| \nabla \phi_m\|_{L^6} 
\| \A \uu_m\|_{L^2}\\
&\leq C \| \D \uu_m \|_{L^2}^\frac12 \| \A \uu_m\|_{L^2}^\frac32 \| \phi_m\|_{H^2}\\
&\leq  \frac{\nu_\ast }{32} \| \A \uu_m\|_{L^2}^2 
+C  \| D \uu_m\|_{L^2}^2 \| \phi_m \|_{H^2}^4 \\
& \leq  \frac{\nu_\ast }{32} \| \A \uu_m\|_{L^2}^2  
+ C \left(1+ C_0 +H_m \right)^3.
\end{split}
\end{equation*}
Owing to \eqref{phiH2e} and \eqref{p-est}, we obtain
\begin{equation*}
\label{I12}
\begin{split}
\left| \frac{\nu_1-\nu_2}{4} \int_{\Omega} \pi_m \nabla \phi_m\cdot \A \uu_m\, \d x \right|
&\leq 
C \| \pi_m\|_{L^3} \| \nabla \phi_m \|_{L^6} \| \A \uu_m \|_{L^2}\\
&\leq C \| \pi_m \|_{L^2}^\frac12 \| \pi_m\|_{H^1}^\frac12 \| \phi_m\|_{H^2} \| \A \uu_m \|_{L^2}\\
&\leq C \| \D \uu_m\|_{L^2}^\frac14 \| \A \uu_m\|_{L^2}^\frac74 \left( 1+ \| \nabla \mu_m\|_{L^2}^2+ \alpha^2 \| \partial_t \phi_m\|_{L^2}^2\right)^\frac12 \\
&\leq  \frac{\nu_\ast }{32} \| \A \uu_m\|_{L^2}^2 
+C \| D \uu_m\|_{L^2}^2 \left( 1+ \| \nabla \mu_m\|_{L^2}^2+ \alpha^2 \| \partial_t \phi_m\|_{L^2}^2\right)^4\\
&\leq  \frac{\nu_\ast }{32} \| \A \uu_m\|_{L^2}^2 
+ C \left( 1+C_0+ H_m \right)^5.
\end{split}
\end{equation*}
By using the Young inequality, it easily follows that
\begin{equation*}
\label{I13}
\begin{split}
\left|  \int_{\Omega} \nabla \partial_t \phi_m \cdot \nabla \Delta \mu_m\, \d x
\right|
\leq \frac{1}{8\varpi_2} \| \nabla \partial_t \phi_m\|_{L^2}^2
+ 2 \varpi_2 \| \mu_m\|_{H^3}^2.
\end{split}
\end{equation*}
Finally, by exploiting \eqref{LADY3}, \eqref{Agmon3d}, \eqref{KORN}, \eqref{phiH2e} and \eqref{H-b}, we infer that
\begin{equation*}
\label{I14}
\begin{split}
\left| \int_{\Omega} \nabla (\uu_m \cdot \nabla \phi_m) \cdot \nabla \Delta \mu_m \, \d x \right|
&\leq C \left( \| \D \uu_m \|_{L^3} \| \nabla \phi_m\|_{L^6} + \| \nabla ^2 \phi_m\|_{L^2} \| \uu_m\|_{L^\infty} \right) \|\nabla \Delta \mu_m\|_{L^2} \\
&\leq C \| \D \uu_m \|_{L^2}^\frac12 \| \A \uu_m \|_{L^2}^\frac12 \| \phi_m\|_{H^2} \| \mu_m\|_{H^3}\\
&\leq \frac{\nu_\ast \varpi_1}{32 \varpi_2} \| \A \uu_m\|_{L^2}^2  
+\frac{1}{6} \| \mu_m\|_{H^3}^2
+ C \| \D \uu_m \|_{L^2}^2 \| \phi_m\|_{H^2}^4\\
&\leq \frac{\nu_\ast \varpi_1}{32 \varpi_2} \| \A \uu_m\|_{L^2}^2  
+\frac{1}{6} \| \mu_m\|_{H^3}^2
+ C\left( 1+C_0+ H_m \right)^3.
\end{split}
\end{equation*}
Combining \eqref{High5} with the above estimates, we arrive at
\begin{equation}
\label{High6}
\begin{split}
\ddt H_m &+ \frac{\rho_\ast}{2} \|\partial_t \uu_m\|_{L^2}^2
+ \frac12 \| \nabla \partial_t \phi_m\|_{L^2}^2 + \left( \frac{\nu_\ast \varpi_1}{4} - \frac{2 \left(\rho^\ast\right)^2 \varpi_1^2}{\rho_\ast} \right) \| \A \uu_m\|_{L^2}^2\\
&+ \left( \frac{\varpi_2}{2}- 2 \varpi_2^2\right) \| \mu_m\|_{H^3}^2
 \leq C \left(1+ C_0+H_m \right)^5,
\end{split}
\end{equation}
where the positive constant $C$depends on $\varpi_1$ and $\varpi_2$, but is independent of $m$, $\alpha$ and $k$. 
Therefore, by setting
$$
\varpi_1= \frac{\rho_\ast \nu_\ast}{16 (\rho^\ast)^2}, \quad \varpi_2= \frac18, 
$$
we deduce the differential inequality 
\begin{equation}
\label{High7}
\ddt H_m
+F_m
\leq C (1+C_0+H_m)^5,
\end{equation}
where
$$
F_m(t)= \frac{\rho_\ast}{2} \|\partial_t \uu_m(t)\|_{L^2}^2
+ \frac12 \| \nabla \partial_t \phi_m(t)\|_{L^2}^2
+ \frac{\varpi_1 \nu_\ast}{8}\| \A \uu_m(t)\|_{L^2}^2
+ \frac{1}{32} \| \mu_m (t)\|_{H^3}^2,
$$
and the constant $C$ is independent of the approximation parameters $\alpha$, $m$ and $k$. Hence, whenever $\widetilde{T}>0$ satisfies
$$
1-4C\widetilde{T}(1+C_0+H_m(0))^4>0,
$$
we infer that
\begin{equation}
C_0+ H_m(t) \leq 
\frac{1+C_0+H_m(0)}{\left( 1-4C t \left(C_1+H_m(0) \right)^4 \right)^\frac14}, \quad \forall \, t \in [0,\widetilde{T}].
\end{equation}
To deduce an estimate of $H_m$ which is independent of $m$, $\alpha$ and $k$, we are left to control $\alpha \| \partial_t \phi_m (0)\|_{L^2}^2$ (cf. definition of $H_m$  and \eqref{H-ab}). To this aim, we first observe that 
$\partial_t \phi_m \in C([0,T];H^1(\Omega))$, $ \mu _m \in C([0,T];H^1(\Omega))$ due to the regularity in Theorem \ref{vCH}. 
By comparison in \eqref{CH-w}$_2$, it follows that $-\Delta \phi_m +\Psi'(\phi_m) \in C([0,T];H^1(\Omega))$. Now, multiplying \eqref{CH-w}$_2$ by $\partial_t \phi_m$ and integrating over $\Omega$, we have
$$
\alpha \| \partial_t \phi_m\|_{L^2}^2+ \left( -\Delta \phi_m +\Psi'(\phi_m), \partial_t \phi_m \right) = (\mu_m, \partial_t \phi_m).
$$
By using \eqref{CH-w}$_1$, we find
$$
\alpha \| \partial_t \phi_m\|_{L^2}^2
+ ( -\Delta \phi_m +\Psi'(\phi_m), \Delta \mu_m - \uu_m \cdot \nabla \phi_m ) 
= (\mu_m, \Delta \mu_m -\uu_m \cdot \nabla \phi_m ).
$$
Integrating by parts, we arrive at
$$
\alpha \| \partial_t \phi_m\|_{L^2}^2
+\| \nabla \mu_m\|_{L^2}^2
= \left( \nabla (-\Delta \phi_m +\Psi'(\phi_m)), \nabla \mu_m - \phi_m \uu_m \right)+
(\nabla \mu_m, \phi_m \uu_m ).
$$
By continuity, we obtain
\begin{align*}
\alpha \| \partial_t \phi_m(0)\|_{L^2}^2
&+\| \nabla \mu_m(0) \|_{L^2}^2\\
&= \left( \nabla (-\Delta \phi_{0,k} +\Psi'(\phi_{0,k})), \nabla \mu_m(0) - \phi_{0,k} \, \uu_m(0) \right)+
(\nabla \mu_m(0), \phi_{0,k} \, \uu_m(0) ),
\end{align*}
which, in turn, implies that 
\begin{equation}
\alpha \| \partial_t \phi_m(0)\|_{L^2}^2
+ \| \nabla \mu_m(0) \|_{L^2}^2
\leq C \| \nabla (-\Delta \phi_{0,k} +\Psi'(\phi_{0,k}))\|_{L^2}^2+
C \| \uu_m(0)\|_{L^2}^2.
\end{equation}
Thus, we conclude from \eqref{ellapp},  \eqref{mukH1}, \eqref{psiH1} and \eqref{H-ab} that
$$
H_m(0)\leq C\left(1+ \| \uu_0\|_{\H^1_\sigma}^2 + \| -\Delta \phi_0 +F'(\phi_0) \|_{H^1}^2 + \| \phi_0\|_{H^1}^2 \right) + \widetilde{C}_0
:= \widetilde{K}_0,
$$
where the constant $C$ is independent of $m$, $\alpha$ and $k$.
Therefore, setting 
$
\widetilde{T}_0=\frac{1}{4C(C_1+ \widetilde{K}_0) )^4},
$
it yields  that
$$
0\leq C_0 + H_m(t) \leq \frac{1+C_0+\widetilde{K}_0}{\left(1- 4 C t \left( C_1+\widetilde{K}_0\right)^4\right)^\frac14}, \quad \forall \, t \in [0,\widetilde{T}_0).
$$
Notice that $\widetilde{T}_0$ is independent of $m$, $\alpha$ and $k$. 
Let us now fix $T_0 \in (0,\widetilde{T}_0)$.
Thanks to \eqref{H-b}, we infer that
\begin{equation}
\label{HE1}
\sup_{t\in [0,T_0]} \| \nabla \uu_m(t)\|_{L^2} + \sup_{t\in[0,T_0]} \| \nabla \mu_m(t)\|_{L^2}
+ \sup_{t\in[0,T_0]} \sqrt{\alpha} \| \partial_t \phi_m(t)\|_{L^2}
\leq K_1,
\end{equation}
where $K_1$ is a positive constant that depends on $E(\uu_0,\phi_0)$, $\| \uu_0\|_{\H^1_\sigma}$, $\| \mu_0\|_{H^1}$, and the parameters of the system, but is independent of $m$, $\alpha$ and $k$.
Recalling \eqref{phiH2e} and \eqref{mu-H1e}, we immediately obtain 
\begin{equation}
\label{HE2}
\sup_{t\in [0,T_0]} \| \phi_m(t)\|_{H^2} +
\sup_{t\in[0,T_0]} \|  \mu_m(t)\|_{H^1}
+ \sup_{t\in[0,T_0]} \| F'(\phi_m(t))\|_{L^2}
\leq K_{2}.
\end{equation}
Integrating \eqref{High5} on $[0,T_0]$, we deduce that
\begin{equation}
\label{HE4}
\int_0^{T_0} 
\|\partial_t \uu_m(\tau)\|_{L^2}^2
+\| \nabla \partial_t \phi_m(\tau)\|_{L^2}^2
+\| \A \uu_m(\tau)\|_{L^2}^2
+\| \mu_m (\tau)\|_{H^3}^2 \, \d \tau \leq K_3.
\end{equation}
Finally, in light of the regularity properties \eqref{HE1} and \eqref{HE4} of the velocity, we observe that the separation property \eqref{phi-As}$_2$ (cf. Theorem \ref{r-vCH}) only depends on $\alpha$ and $k$, but it independent of $m$, namely
\begin{equation}
\label{sp-m}
\phi_m \in L^\infty(\Omega\times (0,T)) : |\phi_m(x,t)|\leq1- \widetilde{\delta} \ \text{a.e. in } \  \Omega\times(0,T_0)
\end{equation}
for some $\widetilde{\delta}=\widetilde{\delta}(\alpha, k)$.

\subsection{Passage to the Limit and Existence of Strong Solutions}
Thanks to the above estimates \eqref{HE1}-\eqref{HE4}, we deduce the following convergences (up to a subsequence) as $m\rightarrow \infty$
\begin{equation}
\label{wl-SS}
\begin{split}
\begin{aligned}
&\uu_m \rightharpoonup \uu_\alpha \quad &&\text{weak-star in } L^\infty(0,T_0;\H^1_\sigma),\\
&\uu_m \rightharpoonup \uu_\alpha \quad  &&\text{weakly in } L^2(0,T_0;H^2)\cap W^{1,2}(0,T_0;\LL^2_\sigma),\\
&\phi_m \rightharpoonup \phi_\alpha \quad  &&\text{weak-star in } L^\infty(0,T_0;H^2(\Omega)),\\
&\phi_m \rightharpoonup \phi_\alpha \quad  &&\text{weakly in } 
W^{1,2}(0,T_0;H^1(\Omega)),\\
&\mu_m  \rightharpoonup \mu_\alpha \quad  &&\text{weak-star in } L^\infty(0,T_0;H^1(\Omega)),\\
&\mu_m  \rightharpoonup \mu_\alpha \quad  &&\text{weakly in } L^2(0,T_0;H^3(\Omega)).
\end{aligned}
\end{split}
\end{equation}
The strong convergences of $\uu_m$ and $\phi_m$ are recovered through the Aubin-Lions lemma, which implies that
\begin{equation}
\label{sl-SS}
\begin{split}
\begin{aligned}
&\uu_m \rightarrow \uu_\alpha \quad &&\text{strongly in } L^2(0,T_0;\H^1_\sigma),\\
&\phi_m \rightarrow \phi_\alpha \quad  &&\text{strongly in } C([0,T_0];W^{1,p}(\Omega)), \ \forall \, p \in [2,6).
\end{aligned}
\end{split}
\end{equation}
As a consequence, we infer that
\begin{equation}
\label{sl-SS2}
\begin{split}
\begin{aligned}
&\rho(\phi_m) \rightarrow \rho(\phi_\alpha),\quad  \nu(\phi_m) \rightarrow \nu(\phi_\alpha)\quad  &&\text{strongly in } C([0,T_0];W^{1,p}(\Omega)),
\end{aligned}
\end{split}
\end{equation}
for all $p \in [2,6)$. Additionally, we have
\begin{equation}
\label{sp-alpha}
\phi_\alpha \in L^\infty(\Omega\times (0,T)) : |\phi_\alpha(x,t)|\leq1-\delta \ \text{a.e. in } \  \Omega\times(0,T_0)
\end{equation}
for some $\delta=\delta(\alpha, k)$. The above properties entail the convergence of the nonlinear terms in \eqref{NS-w} and of the logarithmic potential $\Psi'(\phi)$ in \eqref{CH-w}, thereby we pass to the limit in the Galerkin formulation as $m\rightarrow \infty$ in \eqref{NS-w}-\eqref{CH-w}. The limit solution $(\uu_\alpha,\phi_\alpha)$ satisfies
\begin{equation}
\label{NS-alpha}
\begin{split}
\left ( \rho(\phi_\alpha) \partial_t \uu_\alpha, \ww \right) &+ 
\left( \rho(\phi_\alpha)(\uu_\alpha\cdot \nabla)\uu_\alpha,\ww\right)
- \left( \div( \nu(\phi_\alpha) \D \uu_\alpha),\ww \right) \\
&-\left(\rho'(\phi_\alpha) (\nabla \mu_\alpha \cdot \nabla) \uu_\alpha,\ww\right)
-\left( \mu_\alpha \nabla \phi_\alpha, \ww \right)=0, 
\end{split}
\end{equation}
for all $\ww \in \LL^2_\sigma$, $t \in [0,T_0]$, and  
\begin{equation}
\label{CH-alpha}
\partial_t \phi_\alpha +\uu_\alpha \cdot \nabla \phi_\alpha = \Delta \mu_\alpha,\quad 
\mu_\alpha= \alpha \partial_t \phi_\alpha -\Delta \phi_\alpha+\Psi'(\phi_\alpha)
\quad \text{a.e. in } \ \Omega \times (0,T_0).  
\end{equation}
Moreover, we have
\begin{equation}
\label{bc-alpha}
\begin{cases}
\uu_\alpha=\mathbf{0}, \quad \partial_\n \phi_\alpha=\partial_\n \mu_\alpha=0 \quad &\text{a.e. on } \partial \Omega \times (0,T),\\
\uu_\alpha(\cdot,0)= \uu_{0}, \ \phi(\cdot,0)=\phi_{0,k} \quad &\text{in } \Omega.
\end{cases}
\end{equation}
Next, we proceed with the vanishing viscosity limit in the Cahn-Hilliard equation. Thanks to the lower semicontinuity of the norm, we obtain from \eqref{HE1}-\eqref{HE4} that
\begin{equation}
\label{HE1-alpha}
\esssup_{t\in (0,T_0)} \| \nabla \uu_\alpha(t)\|_{L^2} 
+ \esssup_{t\in (0,T_0)} \| \mu_\alpha(t)\|_{H^1}
+ \esssup_{t\in (0,T_0)} \sqrt{\alpha} \| \partial_t \phi_\alpha(t)\|_{L^2}
\leq K_1,
\end{equation}
\begin{equation}
\label{HE2-alpha}
\esssup_{t\in [0,T_0]} \| \phi_\alpha(t)\|_{H^2} 
+ \esssup_{t\in[0,T_0]} \| F'(\phi_\alpha(t))\|_{L^2}
\leq K_{2},
\end{equation}
and
\begin{equation}
\label{HE4-alpha}
\int_0^{T_0} 
\|\partial_t \uu_\alpha(\tau)\|_{L^2}^2
+\| \nabla \partial_t \phi_\alpha(\tau)\|_{L^2}^2
+\| \A \uu_\alpha(\tau)\|_{L^2}^2
+\| \mu_\alpha (\tau)\|_{H^3}^2 \, \d \tau \leq K_3.
\end{equation} 
Therefore, we can infer that 
\begin{equation}
\label{wl-SS-alpha}
\begin{split}
\begin{aligned}
&\uu_\alpha \rightharpoonup \uu_k \quad &&\text{weak-star in } L^\infty(0,T_0;\H^1_\sigma),\\
&\uu_\alpha \rightharpoonup \uu_k \quad  &&\text{weakly in } L^2(0,T_0;H^2)\cap W^{1,2}(0,T_0;\LL^2_\sigma),\\
&\phi_\alpha \rightharpoonup \phi_k \quad  &&\text{weak-star in } L^\infty(0,T_0;H^2(\Omega)),\\
&\phi_\alpha \rightharpoonup \phi_k \quad  &&\text{weakly in } 
W^{1,2}(0,T_0;H^1(\Omega)),\\
&\mu_\alpha  \rightharpoonup \mu_k \quad  &&\text{weak-star in } L^\infty(0,T_0;H^1(\Omega)),\\
&\mu_\alpha  \rightharpoonup \mu_k \quad  &&\text{weakly in } L^2(0,T_0;H^3(\Omega)).
\end{aligned}
\end{split}
\end{equation}
In a similar manner as above, we have
\begin{equation}
\label{sl-SS-alpha}
\begin{split}
\begin{aligned}
&\uu_\alpha \rightarrow \uu_k \quad &&\text{strongly in } L^2(0,T_0;\H^1_\sigma),\\
&\phi_\alpha \rightarrow \phi_k \quad  &&\text{strongly in } C([0,T_0];W^{1,p}(\Omega)), \\
&\rho(\phi_\alpha) \rightarrow \rho(\phi_k) \quad  &&\text{strongly in } C([0,T_0];W^{1,p}(\Omega)),\\
&\nu(\phi_\alpha) \rightarrow \nu(\phi_k) \quad  &&\text{strongly in } C([0,T_0];W^{1,p}(\Omega)),\\
\end{aligned}
\end{split}
\end{equation}
for all $p \in [2,6)$. In order to pass to the limit in $F'$, we observe that
$$
\phi_\alpha \in L^\infty(\Omega\times (0,T_0)) : |\phi_\alpha(x,t)|<1 \ \text{a.e. in } \  \Omega\times(0,T_0).
$$
Thanks to \eqref{sl-SS-alpha}$_2$, it follows that $\phi_\alpha\rightarrow \phi_k$ almost everywhere in $\Omega \times (0,T)$, and thereby
$$
\phi_k \in L^\infty(\Omega\times (0,T_0)) : |\phi_k(x,t)|<1 \ \text{a.e. in } \  \Omega\times(0,T_0).
$$
Then, we have that $F'(\phi_\alpha) \rightarrow F'(\phi_k)$ almost everywhere in $\Omega \times (0,T)$ and, by Fatou Lemma, $F'(\phi_k)\in L^2(\Omega \times (0,T))$. Owing to this, and by \eqref{HE2-alpha}, we conclude that 
$$
F'(\phi_\alpha) \rightharpoonup F'(\phi_k) \quad  \text{weakly in } 
L^\infty(0,T;\L2).
$$
Thus, letting $\alpha \rightarrow 0$ in \eqref{CH-alpha}-\eqref{NS-alpha}, we obtain
\begin{equation}
\label{NS-k}
\begin{split}
\left ( \rho(\phi_k) \partial_t \uu_k, \ww \right) &+ 
\left( \rho(\phi_k)(\uu_k \cdot \nabla)\uu_k,\ww\right)
- \left( \div( \nu(\phi_k) \D \uu_k),\ww \right) \\
& -\left(\rho'(\phi_k) (\nabla \mu_k \cdot \nabla) \uu_k,\ww\right)
-\left( \mu_k \nabla \phi_k, \ww \right)=0, 
\end{split}
\end{equation}
for all $\ww \in \LL^2_\sigma$, $t \in [0,T_0]$, and  
\begin{equation}
\label{CH-k}
\partial_t \phi_k+\uu_k \cdot \nabla \phi_k= \Delta \mu_k, \quad 
\mu_k= -\Delta \phi_k+\Psi'(\phi_k)
\quad \text{a.e. in } \ \Omega \times (0,T_0),
\end{equation}
together with
\begin{equation}
\label{bc-k}
\begin{cases}
\uu_k=\mathbf{0}, \quad \partial_\n \phi_k=\partial_\n \mu_k=0 \quad &\text{a.e. on } \partial \Omega \times (0,T),\\
\uu_k(\cdot,0)= \uu_{0}, \ \phi(\cdot,0)=\phi_{0,k} \quad &\text{in } \Omega.
\end{cases}
\end{equation}
Finally, since the estimates \eqref{HE1-alpha}-\eqref{HE4-alpha} are independent of $k$, we can further pass to the limit as $k\rightarrow \infty$. The argument readily follows the one above, and so it left to the reader. As a result, we obtain
\begin{equation}
\label{NS-f}
\begin{split}
\left ( \rho(\phi) \partial_t \uu + \rho(\phi)(\uu \cdot \nabla)\uu- \div( \nu(\phi) \D \uu)-\rho'(\phi) (\nabla \mu \cdot \nabla) \uu- \mu \nabla \phi, \ww \right)=0, 
\end{split}
\end{equation}
for all $\ww \in \LL^2_\sigma$, $t \in [0,T_0]$, and  
\begin{equation}
\label{CH-f}
\partial_t \phi+\uu \cdot \nabla \phi= \Delta \mu\quad
\mu= -\Delta \phi+\Psi'(\phi)
\quad \text{a.e. in } \ \Omega \times (0,T_0),
\end{equation}
together with
\begin{equation}
\label{bc-f}
\begin{cases}
\uu=\mathbf{0}, \quad \partial_\n \phi=\partial_\n \mu=0 \quad &\text{a.e. on } \partial \Omega \times (0,T),\\
\uu(\cdot,0)= \uu_{0}, \ \phi(\cdot,0)=\phi_0 \quad &\text{in } \Omega.
\end{cases}
\end{equation}
Recalling the well-known relation
$$
\mu \nabla \phi= - \div( \nabla \phi \otimes \nabla \phi) + \nabla \left( \frac12 |\nabla \phi|^2 + \Psi(\phi) \right),
$$
in a classical way, there exists $P \in L^2(0,T_0;H^1(\Omega))$, $\overline{P}(t)=0$ (see, e.g., \cite{Galdi}) such that
$$
\nabla P=-\rho(\phi) \partial_t \uu - \rho(\phi)(\uu\cdot \nabla)\uu+ \div( \nu(\phi) \D \uu)+\rho'(\phi) \nabla \uu \nabla \mu - 
\div(\nabla \phi \otimes \nabla \phi).
$$
Moreover, exploiting the regularity theory of the Cahn-Hilliard equation with logarithmic potential (see \cite[Lemma 2]{A2009} or \cite[Theorem A.2]{GMT2019}), we deduce that $\phi \in L^\infty(0,T; W^{2,6}(\Omega))$ and $F'(\phi)\in L^\infty(0,T;L^6(\Omega))$.

\section{Proof of Theorem \ref{mr1}. Part two: Uniqueness}
\label{Uni}
\setcounter{equation}{0}
Let $(\uu_1,P_1,\phi_1)$ and $(\uu_2,P_2,\phi_2)$ be two strong solutions to system \eqref{AGG}-\eqref{AGG-bc} defined on the interval $[0,T_0]$ as stated in Theorem \ref{mr1}. We define $\uu=\uu_1-\uu_2$, $P=P_1-P_2$ and $\phi=\phi_1-\phi_2$, which solve
\begin{align}
\label{U-u}
\begin{split}
&\rho(\phi_1)\partial_t \uu + (\rho(\phi_1)-\rho(\phi_2)) \partial_t \uu_2 +
\big(\rho(\phi_1)(\uu_1 \cdot \nabla) \uu_1- \rho(\phi_2)(\uu_2 \cdot \nabla) \uu_2\big)\\
&\quad - \frac{\rho_1-\rho_2}{2}\big( (\nabla \mu_1\cdot \nabla) \uu_1-(\nabla \mu_2 \cdot \nabla) \uu_2 \big)
- \div (\nu(\phi_1)\D\uu) 
-\div( (\nu(\phi_1)-\nu(\phi_2))\D \uu_2)\\
&\quad + \nabla P= - \div(\nabla \phi_1 \otimes \nabla \phi_1 -  \nabla \phi_2\otimes \nabla \phi_2),
\end{split}
\\[10pt]
\label{U-phi}
\begin{split}
&\partial_t \phi +\uu_1\cdot \nabla \phi + \uu \cdot \nabla \phi_2= \Delta \mu,\\
&\mu= -\Delta \phi+\Psi'(\phi_1)- \Psi'(\phi_2),
\end{split}
\end{align}
almost everywhere in $\Omega \times (0,T_0)$. We recall that 
\begin{equation}
\label{uniq-est}
\| \phi_i\|_{L^\infty(0,T_0;W^{2,6}(\Omega))}+ \| \partial_t \phi_i\|_{L^2(0,T_0; H^1(\Omega))} \leq K, \quad i=1, 2,
\end{equation}
where K is a positive constant only depending on $E(\uu_0,\phi_0)$, $\| \uu_0\|_{\H^1_\sigma}$, $\| \mu_0\|_{H^1}$ and $T_0$. As a consequence, we claim that
$$
\| \phi_i\|_{C^{\frac{5}{16}}([0,T_0]; C(\overline{\Omega}))} \leq C K, \quad i=1, 2,
$$
for some constant $C$ depending only on $\Omega$. Indeed, by \eqref{GN-W6}, we have
\begin{align*}
\| \phi_i(t_1)-\phi_i(t_2)\|_{C(\overline{\Omega})}
&\leq C \| \phi_i(t_1)-\phi_i(t_2)\|_{W^{1,4}}\\
&\leq C \| \phi_i(t_1)-\phi_i(t_2)\|_{H^1}^\frac58 
\| \phi_i(t_1)-\phi_i(t_2)\|_{W^{2,6}}^\frac38\\
&\leq C K^\frac38 \left( \int_{t_1}^{t_2} \| \partial_t \phi_i(\tau) \|_{H^1} \, \d \tau \right)^\frac58\\
&\leq C K^\frac38 \| \partial_t \phi_i\|_{L^2(0,T_0; H^1(\Omega))}^\frac58 |t_1-t_2|^\frac{5}{16}, \quad \forall \, t_1,t_2 \in [0,T_0], \, i=1,2.
\end{align*}
In light of the assumption $\| \phi_0\|_{L^\infty}=1-\delta_0$ for some $\delta_0>0$, we infer that
\begin{equation}
\label{sep-prop3}
\| \phi(t)\|_{L^\infty}\leq 1-\frac{\delta_0}{2}, \quad \forall \, t \in [0,T_1], \quad \text{where} \quad  T_1= \left( \frac{\delta_0}{2CK}\right)^\frac{16}{5}. 
\end{equation}
Owing to \eqref{sep-prop3}, it is possible to deduce by elliptic regularity that $\phi \in L^2(0,T_1;H^5(\Omega))$ and $\partial_t \mu \in L^2(0,T_1;(H^1(\Omega))')$.

Next, multiplying \eqref{U-u} by $\uu$ and integrating over $\Omega$, we find
\begin{equation}
\label{U-1}
\begin{split}
&\frac12 \ddt \int_{\Omega} \rho(\phi_1) |\uu|^2 \, \d x 
+ \int_{\Omega} \nu(\phi_1) |\D \uu|^2 \, \d x
\\
&= - \int_{\Omega} (\rho(\phi_1)-\rho(\phi_2) ) \partial_t \uu_2 \cdot \uu \, \d x
-\int_{\Omega} \rho(\phi_1) (\uu \cdot \nabla) \uu_2 \cdot \uu \, \d x\\
&\quad 
- \int_{\Omega} (\rho(\phi_1)-\rho(\phi_2))(\uu_2\cdot \nabla)\uu_2 \cdot \uu\, \d x
+\frac{\rho_1-\rho_2}{2} \int_{\Omega} ((\nabla \mu \cdot\nabla) \uu_2)  \cdot \uu \, \d x \\
&\quad -\int_{\Omega} (\nu(\phi_1)-\nu(\phi_2)) \D \uu_2 : \nabla \uu \, \d x
+ \int_{\Omega} (\nabla \phi_1 \otimes \nabla \phi + \nabla \phi\otimes \nabla \phi_2): \nabla \uu \, \d x\\
&= \sum_{i=1}^6 Z_i.
\end{split}
\end{equation} 
Here we have used that 
$$
-\int_{\Omega} \partial_t \rho(\phi_1) \frac{|\uu|^2}{2}\, \d x 
+\int_{\Omega} \rho(\phi_1) \uu_1\cdot \nabla \frac{|\uu|^2}{2} \, \d x 
-\frac{\rho_1-\rho_2}{2} \int_{\Omega} \nabla \mu_1\cdot \nabla \frac{|\uu|^2}{2} \, \d x=0.  
$$
Taking the gradient of \eqref{U-phi}$_1$, multiplying by $\nabla \Delta \phi$ and integrating over $\Omega$, we obtain
\begin{align*}
\frac12 \ddt \| \Delta \phi\|_{L^2}^2 + \| \Delta^2 \phi\|_{L^2}^2 
&= \int_{\Omega} \uu_1\cdot \nabla \phi \Delta^2 \phi \, \d x+
\int_{\Omega} \uu \cdot \nabla \phi_2 \Delta^2 \phi \, \d x
+ \int_{\Omega} \Delta( \Psi'(\phi_1)-\Psi'(\phi_2)) \Delta^2 \phi \, \d x\\
&= \sum_{i=7}^9 Z_i.
\end{align*}
Therefore, we arrive at 
\begin{align*}
\ddt \left( \frac12 \int_{\Omega} \rho(\phi_1) |\uu|^2 \, \d x  + \frac12 \| \Delta \phi\|_{L^2}^2 \right)
+ \int_{\Omega} \nu(\phi_1) |\D \uu|^2 \, \d x 
+ \| \Delta^2 \phi\|_{L^2}^2 = \sum_{i=1}^9 Z_i.
\end{align*}
Arguing in a similar way as in \cite[Section 6]{G2021}, it is easily seen that
\begin{equation*}
|Z_1+Z_2+Z_3+Z_5+Z_6|
\leq \frac{\nu_\ast}{2}\| \D \uu\|_{L^2}^2
+C \left(1+ \| \uu_2\|_{H^2}^2+ \| \partial_t \uu_2\|_{L^2}^2 \right)
\left( \|\uu \|_{L^2}^2 + \| \Delta \phi\|_{L^2}^2 \right).
\end{equation*}
By \eqref{uniq-est} and \eqref{sep-prop3}, together with Sobolev embeddings, we find
\begin{align*}
|Z_4| & \leq \int_{\Omega} \left| \left( \nabla \Delta \phi\cdot \nabla\right) \uu_2 \cdot \uu \right| \, \d x 
+ \int_{\Omega}  \left| \left( \nabla \left(\Psi'(\phi_1)-\Psi'(\phi_2) \right) \cdot \nabla\right) \uu_2 \cdot \uu \right| \, \d x\\
&\leq \| \nabla \Delta \phi\|_{L^6} \| \nabla \uu_2\|_{L^3} \| \uu\|_{L^2} + \| \Psi''(\phi_1)\|_{L^\infty} \| \nabla \phi\|_{L^6}
\| \nabla \uu_2\|_{L^3} \| \uu\|_{L^2}\\
&\quad + \left( \| \Psi'''(\phi_1)\|_{L^\infty}
+ \| \Psi'''(\phi_2)\|_{L^\infty}\right) \| \phi\|_{L^\infty} \| \nabla \phi_2\|_{L^\infty} \| \nabla \uu_2\|_{L^2} \| \uu\|_{L^2}\\
& \leq \frac16 \| \Delta^2 \phi\|_{L^2}^2+ C \| \nabla \uu_2\|_{L^3}^2 
\| \uu\|_{L^2}^2 + C\left(1+ \|\nabla \uu_2\|_{L^3} \right) \left( \| \uu\|_{L^2}^2 + \| \Delta \phi\|_{L^2}^2 \right).
\end{align*}
As to the remaining terms, by using \eqref{uniq-est} and \eqref{sep-prop3} once more, we have
\begin{align*}
\left| Z_7+ Z_8\right|
& \leq \| \uu_1\|_{L^3} \| \nabla \phi\|_{L^6} \| \Delta^2 \phi\|_{L^2} + \| \uu\|_{L^2} \| \nabla \phi_2\|_{L^\infty} \| \Delta^2 \phi\|_{L^2}\\
&\leq \frac16 \|\Delta^2 \phi \|_{L^2}^2+ C \left( \| \uu\|_{L^2}^2 +\| \Delta \phi\|_{L^2}^2\right),
\end{align*}
and 
\begin{align*}
|Z_9|&\leq \int_{\Omega} \left| \left( \Psi''(\phi_1) \Delta \phi + \left( \Psi''(\phi_1)-\Psi''(\phi_2) \right) \Delta \phi_2 \right) \Delta^2 \phi \right| \, \d x \\
& \quad +\int_{\Omega} \left| \left( \Psi'''(\phi_1) \left( |\nabla \phi_1|^2- |\nabla \phi_2|^2 \right) + \left( \Psi'''(\phi_1)-\Psi'''(\phi_2) \right) |\nabla \phi_2|^2  \right) \Delta^2 \phi \right| \, \d x\\
&\leq C \| \Delta \phi\|_{L^2} \| \Delta^2 \phi\|_{L^2} + C \left( \| \Psi'''(\phi_1)\|_{L^\infty}+ \| \Psi'''(\phi_2)\|_{L^\infty}\right) \| \phi\|_{L^\infty} \| \Delta \phi_2\|_{L^2} \| \Delta^2 \phi\|_{L^2}\\
&\quad + C \left( \| \nabla \phi_1\|_{L^\infty)}+\| \nabla \phi_2 \|_{L^\infty} \right) \| \nabla \phi\|_{L^2} \| \Delta^2 \phi\|_{L^2} \\
&\quad + \left( \| \Psi''''(\phi_1)\|_{L^\infty}+ \| \Psi''''(\phi_2)\|_{L^\infty}\right) \| \phi\|_{L^\infty} \|\nabla \phi_2\|_{L^\infty}^2 \| \Delta^2 \phi\|_{L^2}\\
&\leq \frac16 \|\Delta^2 \phi \|_{L^2}^2+ C \| \Delta \phi\|_{L^2}^2.
\end{align*}
In conclusion, we find the differential inequality
\begin{align*}
\ddt \left( \frac12 \int_{\Omega} \rho(\phi_1) |\uu|^2 \, \d x  + \frac12 \| \Delta \phi\|_{L^2}^2 \right) &
+ \frac{\nu_\ast}{2} \| \D \uu \|_{L^2}^2 
+ \frac12 \| \Delta^2 \phi\|_{L^2}^2 \\
\leq 
C(K) &\left( 1 + \| \uu_2\|_{H^2}^2+ \| \partial_t \uu_2\|_{L^2}^2 \right)
\left( \|\uu \|_{L^2}^2 + \| \Delta \phi \|_{L^2}^2 \right).
\end{align*}
An application of the Gronwall lemma implies the desired uniqueness of strong solutions on the time interval $[0,T_1]$.
 
\section{Proof of Theorem \ref{stab}: Stability}
\label{Sta}
\setcounter{equation}{0}

Let $(\uu,P,\phi)$ and $(\uu_H,P_H,\phi_H)$ be the strong solutions to the AGG model with density $\rho(\phi)$ and to the model H with constant density $\overline{\rho}$, respectively, defined on a common interval $[0,T_0]$. We recall that the existence of $(\uu_H,P_H,\phi_H)$ fulfilling the same regularity properties of $(\uu,P,\phi)$, as stated in Theorem \eqref{mr1}, has been proven in \cite[Theorem 5.1]{GMT2019}. For simplicity, we assume that the viscosity function is given by $\nu(s)= \nu_1 \frac{1+s}{2}+\nu_2 \frac{1-s}{2}$ (cf. \eqref{Jrhonu}) for both systems.
We define $\vv=\uu-\uu_H$, $Q=P-P_H$, $\varphi=\phi-\phi_H$, and the difference of the chemical potentials $w= \mu- \mu_H$. They clearly solve the problem
\begin{align}
\label{St1}
\begin{split}
&\left( \frac{\rho_1+\rho_2}{2}\right )\partial_t \vv 
+ \left( \frac{\rho_1-\rho_2}{2} \phi \right) \partial_t \uu
+ \left(\frac{\rho_1+\rho_2}{2}-\overline{\rho}\right) \partial_t \uu_H
+ \left( \rho(\phi)(\uu \cdot \nabla) \uu- \overline{\rho}(\uu_H \cdot \nabla) \uu_H \right)\\
&\quad - \left(\frac{\rho_1-\rho_2}{2}\right) \big( (\nabla \mu\cdot \nabla) \uu \big)
- \div\left(\nu(\phi) \D\vv \right) 
-\div\left( (\nu(\phi)-\nu(\phi_H)) \D \uu_H \right)\\
&\quad + \nabla Q= - \div(\nabla \phi \otimes \nabla \phi -
 \nabla \phi_H \otimes \nabla \phi_H),
\end{split}
\\[10pt]
\label{St2}
\begin{split}
&\partial_t \varphi +\uu\cdot \nabla \varphi + \vv \cdot \nabla \phi_H= \Delta w,\\
&w= -\Delta \varphi+\Psi'(\phi)- \Psi'(\phi_H),
\end{split}
\end{align}
almost everywhere in $\Omega \times (0,T_0)$. In addition, we have the boundary and initial conditions
\begin{align}
\label{St-bcic}
\vv=\mathbf{0}, \quad \partial_\n\varphi=\partial_\n w=0 \quad \text{on } \ \partial \Omega \times (0,T), 
\quad
\vv(\cdot, 0)=\mathbf{0}, \quad \varphi(\cdot, 0)=0 \quad \text{in } \ \Omega.
\end{align}

Multiplying \eqref{St1} by $\A^{-1} \vv$ and integrating over $\Omega$, we obtain
\begin{align*}
\left( \frac{\rho_1+\rho_2}{4}\right) \ddt \| \vv\|_{\sharp}^2
& + \int_{\Omega} \nu(\phi) \D \vv : \nabla \A^{-1} \vv \, \d x =
-\int_{\Omega} \left( \frac{\rho_1-\rho_2}{2} \phi \right) \partial_t \uu \cdot \A^{-1}\vv \, \d x \\
&- \int_{\Omega}    \left(\frac{\rho_1+\rho_2}{2}-\overline{\rho}\right) \partial_t \uu_H \cdot \A^{-1}\vv \, \d x
- \int_{\Omega} \big(\rho(\phi)(\uu \cdot \nabla) \uu- \overline{\rho}(\uu_H \cdot \nabla) \uu_H \big)\cdot  \A^{-1}\vv \, \d x\\
& + \int_{\Omega} \left(\frac{\rho_1-\rho_2}{2}\right) \big( (\nabla \mu\cdot \nabla) \uu \big) \cdot \A^{-1}\vv \, \d x
- \int_{\Omega} (\nu(\phi)-\nu(\phi_H)) \D \uu_H : \nabla \A^{-1}\vv \, \d x\\
& + \int_{\Omega} \nabla \phi \otimes \nabla \phi -
 \nabla \phi_H \otimes \nabla \phi_H : \nabla \A^{-1}\vv \, \d x.
\end{align*}
Following \cite[proof of Theorem 3.1]{GMT2019}, we infer that
\begin{equation}
\int_{\Omega} \nu(\phi) \D \vv : \nabla \A^{-1} \vv \, \d x \geq \frac{\nu_\ast}{2} \| \uu\|_{L^2}^2 - \int_{\Omega} \nu'(\phi) \D \A^{-1} \vv \nabla \phi \cdot \vv \, \d x +\frac12 \int_{\Omega} \nu'(\phi) \nabla \phi\cdot \vv \, \Pi \, \d x,
\end{equation}
where $\Pi \in L^\infty(0,T_0; H^1(\Omega))$ is such that $-\Delta \A^{-1}\vv+\nabla \Pi= \vv$ a.e. in $\Omega \times (0,T_0)$. In addition, it fulfills the estimates 
\begin{equation}
\label{PES}
\| \Pi\|_{L^2}\leq C \| \nabla \A^{-1}\vv\|_{L^2}^\frac12 
\| \vv\|_{L^2}^\frac12, \quad \| \Pi\|_{H^1}\leq C \| \vv\|_{L^2}.
\end{equation}
Therefore, we are led to
\begin{equation}
\label{S-1}
\begin{split}
&\left( \frac{\rho_1+\rho_2}{4}\right) \ddt \| \vv\|_{\sharp}^2
 +\frac{\nu_\ast}{2}\|\vv\|_{L^2}^2 \\
 & =
-\int_{\Omega} \left( \frac{\rho_1-\rho_2}{2} \phi \right) \partial_t \uu \cdot \A^{-1}\vv \, \d x 
- \int_{\Omega}  \left(\frac{\rho_1+\rho_2}{2}-\overline{\rho}\right) \partial_t \uu_H \cdot \A^{-1}\vv \, \d x\\
&\quad - \int_{\Omega} \big(\rho(\phi)(\uu \cdot \nabla) \uu- \overline{\rho}(\uu_H \cdot \nabla) \uu_H \big)\cdot  \A^{-1}\vv \, \d x
 + \int_{\Omega} \left(\frac{\rho_1-\rho_2}{2}\right) \big( (\nabla \mu\cdot \nabla) \uu \big) \cdot \A^{-1}\vv \, \d x\\
&\quad - \int_{\Omega} (\nu(\phi)-\nu(\phi_H)) \D \uu_H : \nabla  \A^{-1}\vv \, \d x + \int_{\Omega} \nabla \phi \otimes \nabla \phi -
 \nabla \phi_H \otimes \nabla \phi_H : \nabla \A^{-1}\vv \, \d x\\
&\quad
 + \int_{\Omega} \nu'(\phi) \D \A^{-1} \vv \nabla \phi \cdot \vv \, \d x -\frac12 \int_{\Omega} \nu'(\phi) \nabla \phi\cdot \vv \, \Pi \, \d x.
\end{split}
\end{equation}
On the other hand, multiplying \eqref{St2}$_2$ by $A^{-1}\varphi$, where $A$ is the Laplace operator with homogeneous Neumann boundary conditions, and integrating over $\Omega$, we get (see \cite[Proof of Theorem 3.1]{GMT2019} for more details)
\begin{equation}
\label{s-CH}
\begin{split}
\frac12 \ddt \| \varphi\|_{\ast}^2 + \frac12 \| \nabla \varphi\|_{L^2}^2
\leq C\| \varphi\|_{\ast}^2 + \int_{\Omega} \varphi \, \uu \cdot \nabla A^{-1} \varphi \, \d x  
+ \int_{\Omega} \phi_{H} \, \vv \cdot \nabla A^{-1}\varphi \, \d x.
\end{split}
\end{equation}
We proceed with the estimate of the terms on the right-hand side of \eqref{S-1} and \eqref{s-CH}. To this aim, we will exploit the following bounds on the solution
\begin{equation}
\begin{split}
\|(\uu, \uu_H) \|_{L^\infty(0,T_0;\H^1_\sigma)\cap L^2(0,T_0;\H_\sigma^2(\Omega))\cap W^{1,2}(0,T_0;\LL^2_\sigma)}\leq K_0,\\
\| (\phi, \phi_H)\|_{L^\infty(0,T_0;W^{2,6}(\Omega))}+\| \nabla \mu\|_{L^\infty(0,T_0;\L2)}\leq K_0,
\end{split}
\end{equation}
where $K_0$ is a constant depending on the norms of the initial conditions.
Exploiting this estimates, we have
\begin{equation*}
\begin{split}
\left| \int_{\Omega} \left( \frac{\rho_1-\rho_2}{2} \phi \right) \partial_t \uu \cdot \A^{-1}\vv \, \d x  \right| 
&\leq 
\left| \frac{\rho_1-\rho_2}{2} \right| \| \phi\|_{L^\infty} 
\| \partial_t \uu\|_{L^2} \| \A^{-1} \vv \|_{L^2}\\
&\leq C\| \vv \|_{\sharp}^2 + C \left| \frac{\rho_1-\rho_2}{2} \right|^2 
 \| \partial_t \uu\|_{L^2}^2, 
\end{split}
\end{equation*}
and
\begin{equation*}
\begin{split}
\left|  \int_{\Omega}  \left( \frac{\rho_1+\rho_2}{2}-\overline{\rho}\right) \partial_t \uu_H\cdot \A^{-1}\vv \, \d x \right| 
\leq  C\|\vv\|_{\sharp}^2
+ C \left| \frac{\rho_1+\rho_2}{2}- \overline{\rho}\right|^2  \| \partial_t \uu_H\|_{L^2}^2.
\end{split}
\end{equation*}
By Sobolev embedding, we find
\begin{equation*}
\begin{split}
&\left|  \int_{\Omega}  \big(\rho(\phi)(\uu \cdot \nabla) \uu- \overline{\rho}(\uu_H \cdot \nabla) \uu_H \big)\cdot  \A^{-1}\vv \, \d x \right|\\
&\leq \left| \int_{\Omega}  \rho(\phi)(\vv \cdot \nabla) \uu \cdot  \A^{-1}\vv \, \d x\right|
+ \left| \int_{\Omega}  \rho(\phi)(\uu_H \cdot \nabla) \vv\cdot  \A^{-1}\vv \, \d x\right| 
+\left| \int_{\Omega}  \big( \rho(\phi)-\overline{\rho}\big) (\uu_H \cdot \nabla) \uu_H \cdot  \A^{-1}\vv \, \d x\right|\\
&\leq \rho^\ast \| \vv\|_{L^2} \|\nabla \uu \|_{L^6}\| \A^{-1}\vv\|_{L^3}
+ \left|\int_{\Omega} \rho(\phi) ( \uu_H \cdot \nabla) \A^{-1}\vv \cdot \vv\, \d x 
+ \int_{\Omega} \rho'(\phi) (\nabla \phi\cdot \uu_H) \left( \vv \cdot \A^{-1}\vv\right) \, \d x \right|\\
&\quad + \|\rho(\phi)-\overline{\rho}\|_{L^\infty} 
\| \uu_H\|_{L^6} \| \nabla \uu_H\|_{L^2}\| \A^{-1} \vv \|_{L^3} \\
&\leq  \frac{\nu_\ast}{16}\| \vv\|_{L^2}^2
+C \left(1+\| \uu\|_{H^2}^2 \right) \|\vv \|_{\sharp}^2+
\rho^\ast \| \nabla \A^{-1}\vv\|_{L^2} \| \uu_H\|_{L^\infty} 
\| \vv\|_{L^2} \\
&\quad + \left| \frac{\rho_1-\rho_2}{2} \right| \|\nabla \phi\|_{L^\infty} 
\| \uu_H\|_{L^6} \| \vv\|_{L^2} \| \A^{-1}\vv\|_{L^3}+
C(K_0)\left( \left| \frac{\rho_1-\rho_2}{2}\right|^2 + \left| \frac{\rho_1+\rho_2}{2}- \overline{\rho}\right|^2 \right)\\
&\leq  \frac{\nu_\ast}{8}\| \vv\|_{L^2}^2
+C(K_0) \left(1+\| \uu\|_{H^2}^2 + \| \uu_H\|_{H^2}^2\right) \|\vv \|_{\sharp}^2+ C(K_0)\left( \left| \frac{\rho_1-\rho_2}{2}\right|^2 + \left| \frac{\rho_1+\rho_2}{2}- \overline{\rho}\right|^2 \right),
\end{split}
\end{equation*}
and
\begin{equation*}
\begin{split}
\left| \int_{\Omega} \left(\frac{\rho_1-\rho_2}{2}\right) \big( (\nabla \mu\cdot \nabla) \uu \big) \cdot \A^{-1}\vv \, \d x \right|
&\leq 
\left| \frac{\rho_1-\rho_2}{2}\right| \| \nabla \mu\|_{L^2} \| \nabla \uu\|_{L^3} \| \A^{-1}\vv\|_{L^6}\\
&\leq C \| \vv \|_{\sharp}^2 + C(K_0) \left| \frac{\rho_1-\rho_2}{2}\right|^2 
 \| \nabla \uu\|_{L^3}^2.
\end{split}
\end{equation*}
In a similar way as in \cite[Proof of Theorem 5.1]{GMT2019}, we obtain
\begin{equation*}
\begin{split}
\left| \int_{\Omega} (\nu(\phi)-\nu(\phi_H)) \D \uu_H : \nabla \A^{-1}\vv \, \d x \right| 
&\leq 
C \| \varphi\|_{L^6}\| \D \uu_H\|_{L^3} \| \nabla \A^{-1}\vv \|_{L^2}\\
&\leq  \frac{1}{6} \| \nabla \varphi\|_{L^2}^2+
C \| \uu_H\|_{H^2}^2
\|\vv \|_{\sharp}^2,
\end{split}
\end{equation*}
\begin{equation*}
\begin{split}
\left| \int_{\Omega} \left(\nabla \phi \otimes \nabla \phi -
 \nabla \phi_H \otimes \nabla \phi_H\right) : \nabla \A^{-1}\vv \, \d x
 \right| 
 &\leq \left( \| \nabla \phi\|_{L^\infty}+ \| \nabla \phi_H\|_{L^\infty}\right) \| \nabla \varphi \|_{L^2}\| \nabla \A^{-1}\vv\|_{L^2}\\
 &\leq \frac{1}{6} \| \nabla \varphi\|_{L^2}^2+
C(K_0) \|\vv \|_{\sharp}^2,
\end{split}
\end{equation*}
\begin{equation*}
\begin{split}
\left|  \int_{\Omega} \nu'(\phi) \D \A^{-1} \vv \nabla \phi \cdot \vv \, \d x 
\right| 
&\leq C \| \D \A^{-1}\vv \|_{L^2} \| \nabla \phi\|_{L^\infty} \| \vv\|_{L^2} \leq \frac{\nu_\ast}{8}\| \vv\|_{L^2}^2
+C(K_0) \|\vv \|_{\sharp}^2,
\end{split}
\end{equation*}
\begin{equation*}
\begin{split}
\left| \frac12 \int_{\Omega} \nu'(\phi) \left(\nabla \phi\cdot \vv\right) \Pi \, \d x \right| 
& \leq C \| \nabla \phi\|_{L^\infty} \| \vv\|_{L^2} \| \Pi\|_{L^2}\leq  \frac{\nu_\ast}{8}\| \vv\|_{L^2}^2 + C(K_0)\|\vv \|_{\sharp}^2,
\end{split}
\end{equation*}
\begin{equation*}
\begin{split}
\left|
\int_{\Omega} \varphi \, \uu \cdot \nabla A^{-1} \varphi \, \d x  
\right| \leq \frac16 \| \nabla \varphi\|_{L^2}^2 + C\| \uu\|_{H^2(\Omega)}^2
\| \varphi\|_{\ast}^2,
\end{split}
\end{equation*}
\begin{equation*}
\begin{split}
\left| \int_{\Omega} \phi_{H} \, \vv \cdot \nabla A^{-1}\varphi \, \d x \right| \leq \frac{\nu_\ast}{8}\| \vv\|_{L^2}^2 +C \| \varphi\|_{\ast}^2.
\end{split}
\end{equation*}
Collecting the above estimates together, we find the differential inequality
\begin{align*}
\ddt \bigg( \left( \frac{\rho_1+\rho_2}{4}\right)\| \vv\|_{\sharp}^2+ \frac12 \|\varphi\|_{\ast}^2 \bigg) \leq f_1(t) \big( \| \vv\|_{\sharp}^2+  \|\varphi\|_{\ast}^2\big) + f_2(t)  \Big(  \Big| \frac{\rho_1-\rho_2}{2}\Big|^2 + \Big| \frac{\rho_1+\rho_2}{2}- \overline{\rho}\Big|^2 \Big),
\end{align*}
where
\begin{align*}
&f_1(t)= C(K_0) \left(1+\| \uu_{H}\|_{H^2}^2+ \| \uu\|_{H^2}^2\right), \\
&f_2(t)= C(K_0) \left( 1+ \| \partial_t \uu_H\|_{L^2}^2+\| \uu_H\|_{H^2}^2
+\| \partial_t \uu\|_{L^2}^2+\| \uu\|_{H^2}^2 \right).
\end{align*}
Here, the positive constant $C$ depends on the norm of the initial data and the time $T_0$. By using the Gronwall lemma, together with the initial conditions \eqref{St-bcic}, we infer that 
$$
\|\vv(t)\|_{\sharp}^2 +  \| \varphi(t)\|_{\ast}^2 
\leq \frac{\left(  \left| \frac{\rho_1-\rho_2}{2}\right|^2 + \left| \frac{\rho_1+\rho_2}{2}- \overline{\rho}\right|^2 \right)}{\min \lbrace \frac{\rho_1+\rho_2}{4}, \frac12 \rbrace }  \int_0^t \mathrm{e}^{\int_s^t f_1(r)\, \d r} f_2(s) \, \d s, \quad \forall \, t\in [0,T_0].  
$$
Thus, the above inequality implies that
$$
\| \uu(t)-\uu_H(t)\|_{(\H^1_\sigma)'}+ \| \phi(t)-\phi_H(t)\|_{(H^1)'}
\leq \frac{C(K_0)}{\min \lbrace \sqrt{\rho_\ast}, 1 \rbrace } \left(  \left| \frac{\rho_1-\rho_2}{2}\right| + \left| \frac{\rho_1+\rho_2}{2}- \overline{\rho}\right| \right), \quad  \forall \, t \in [0,T_0],
$$
where the positive constant $C(K_0)$ depends on the norm of the initial data, the time $T_0$ and the parameters of the systems.

\appendix
\section{On the convective Viscous Cahn-Hilliard system}
\label{App-0}
\setcounter{equation}{0}

\noindent
Given $\alpha>0$ and an incompressible velocity field $\uu$, we consider the  convective Viscous Cahn-Hilliard (cVCH) system
\begin{equation}
\label{vCH}
\partial_t \phi + \uu \cdot \nabla \phi= \Delta \mu, \quad
\mu= \alpha \partial_t \phi - \Delta \phi+ \Psi'(\phi)
\quad \text{in } \Omega \times (0,T),
\end{equation}
with boundary and initial conditions
\begin{equation}
\label{vCH-2}
\partial_\n\phi=\partial_\n \mu=0 \quad \text{on } \partial \Omega \times 
(0,T), \quad 
\phi(\cdot,0)= \phi_0 \quad \text{in }  \Omega.
\end{equation}
We observe that \eqref{vCH} can be rewritten as 
$$
\partial_t (\phi- \alpha \Delta \phi)+ \uu \cdot \nabla \phi= \Delta (-\Delta \phi + F'(\phi)- \theta_0 \phi) \quad \text{in } \Omega \times (0,T).
$$
We state well-posedness and regularity results for system \eqref{vCH}.
The aim of this Appendix is to extend the analysis performed in \cite{MZ} to the convective case under minimal assumptions on the velocity field. In particular, we focus on the regularity of the chemical potential.

\begin{theorem}
\label{r-vCH}
Assume that 
$\uu \in L^\infty(0,T; \LL_\sigma^2(\Omega) \cap L^3(\Omega))$,  
$\phi_0 \in H^1(\Omega)\cap L^\infty(\Omega)$ such that $\| \phi_0\|_{L^\infty}\leq 1$ and $|\overline{\phi_0}|<1$.
Then, there exists a unique a weak solution to \eqref{vCH}-\eqref{vCH-2} such that
\begin{align}
\label{vCH-weak}
\begin{split}
&\phi \in L^\infty(0,T;H^1(\Omega) \cap L^\infty(\Omega)) : |\phi(x,t)| < 1 \ \text{a.e. in } \  \Omega\times(0,T),\\
& \phi \in L^2(0,T;H^2(\Omega))\cap W^{1,2}(0,T;L^2(\Omega)),\\
&\mu \in L^2(0,T;H^2(\Omega)), \quad F'(\phi) \in L^2(0,T;L^2(\Omega)),
\end{split}
\end{align}
which satisfies \eqref{vCH} almost everywhere in $\Omega \times (0,T)$, \eqref{vCH-2} almost everywhere on $\partial \Omega \times (0,T)$ and $\phi (\cdot, 0)= \phi_0(\cdot)$ in $\Omega$.
In addition, the following regularity results hold:
\smallskip

\begin{itemize}
\item[(R1)] If $-\Delta \phi_0+F'(\phi_0) \in L^2(\Omega)$ and $\partial_t \uu \in L^\frac43(0,T;L^1(\Omega))$, we have
\begin{align*}
&\partial_t \phi \in L^\infty(0,T;L^2(\Omega))\cap L^2(0,T;H^1(\Omega)),\\
&\phi \in L^\infty(0,T;H^2(\Omega)), \quad \mu \in L^\infty(0,T;H^2(\Omega)).
\end{align*}

\item[(R2)] Let the assumptions of (R1) hold. Suppose that $\| \phi_0\|_{L^\infty}\leq 1-\delta_0$,  for some $\delta_0 \in (0,1)$. Then, there exists $\delta>0$ such that
 \begin{equation}
\label{SP}
\max_{(x,t)\in \Omega \times (0,T)} |\phi(x,t)|\leq 1-\delta,
\end{equation}
and
$$
\phi \in L^2(0,T;H^3(\Omega)).
$$
\item[(R3)] Let the assumption of (R2) hold. Suppose that 
$\phi_0 \in H^3(\Omega)$ such that $\partial_\n \phi=0$ on $\partial \Omega$, and 
$\partial_t \uu \in L^2(0,T;L^\frac65(\Omega))$, we have
\begin{align*}
&\partial_t \phi \in L^\infty(0,T; H^1(\Omega))\cap L^2(0,T;H^2(\Omega)),\\
&\phi \in L^\infty(0,T;H^3(\Omega))\cap L^2(0,T;H^4(\Omega)),\\
& \partial_t^2 \phi \in L^2(0,T;L^2(\Omega)), \quad \partial_t \mu \in
L^2(0,T;L^2(\Omega)).
\end{align*}
\end{itemize}
\end{theorem}

\begin{proof} The proof is divided in several parts. We notify the reader that the estimates herein proved are not independent of the viscous parameter $\alpha$.
\smallskip

\noindent
\textbf{Existence.} The existence of a weak solution satisfying \eqref{vCH-weak} is proved in a classical way\footnote{The interested reader might exploit the combination of the Galerkin method with the approximation of the logarithmic potential by smooth potentials.}. We proceed here by proving the basic {\it energy} estimates. First, we observe that, by integrating \eqref{vCH}$_1$ over $\Omega$ and using the boundary conditions, we have 
\begin{equation}
\label{E-mass}
\overline{\phi}(t)= \overline{\phi_0} \quad \text{and} \quad 
\overline{\partial_t \phi}(t)=0 \quad \forall \, t \in [0,T].
\end{equation}
Multiplying \eqref{vCH}$_1$ by $\mu$, integrating over $\Omega$, using the boundary conditions \eqref{vCH-2} and \cite[Lemma 4.3, Ch. IV]{SHOW1997}, we find
$$
\ddt \left( \int_{\Omega} \frac12 |\nabla \phi|^2 + \Psi(\phi) \, \d x \right)+
\| \nabla \mu\|_{L^2}^2 + \alpha \| \partial_t \phi\|_{L^2}^2 = \int_{\Omega} \phi \,  \uu \cdot \nabla \mu \, \d x.
$$
By the H\"{o}lder inequality and the boundedness of $\phi$, we simply obtain
$$
\ddt \left( \int_{\Omega} \frac12 |\nabla \phi|^2 + \Psi(\phi) \, \d x \right)+
\frac12 \| \nabla \mu\|_{L^2}^2 + \alpha \| \partial_t \phi\|_{L^2}^2 \leq  \frac12 \|\uu \|_{L^2}^2.
$$
Thus, integrating over $[0,T]$ and using the continuity of $\Psi$, we have
\begin{equation}
\label{E-en1}
\begin{split}
\| \nabla \phi\|_{L^\infty(0,T;L^2(\Omega))}+ 
\| \nabla \mu\|_{L^2(0,T;L^2(\Omega))}&+
\| \partial_t \phi\|_{L^2(0,T;L^2(\Omega))}\\
&\leq C_\alpha \big(  \sqrt{E_{\text{free}}(\phi_0)}+ \| \uu\|_{L^2(0,T;L^2(\Omega))}\big).
\end{split}
\end{equation}
In light of \eqref{normH1-2} and \eqref{E-mass}, we infer that 
\begin{equation}
\label{E-en2}
\| \phi\|_{L^\infty(0,T;H^1(\Omega))}\leq C_\alpha \big(  \sqrt{E_{\text{free}}(\phi_0)}+ \| \uu\|_{L^2(0,T;L^2(\Omega))}+ |\overline{\phi_0}|\big).
\end{equation}
Now, multiplying \eqref{vCH}$_2$ by $-\Delta \phi$ and integrating over $\Omega$, we get
\begin{align*}
\frac{\alpha}{2}\ddt \| \nabla \phi\|_{L^2}^2 +\| \Delta \phi\|_{L^2}^2+ \int_{\Omega} -F'(\phi) \Delta \phi\, \d x
= \int_{\Omega} \nabla \mu \cdot \nabla \phi\, \d x + \theta_0 \| \nabla \phi\|_{L^2}^2.
\end{align*}
The second term on the left-hand side is clearly positive by monotonicity. Then, using \eqref{E-en2} we obtain
\begin{equation}
\int_0^T \| \Delta \phi(\tau)\|_{L^2}^2 \,\d \tau
\leq \frac{\alpha}{2}\| \nabla \phi_0\|_{L^2}^2
+ C_\alpha (1+T) \big(  \sqrt{E_{\text{free}}(\phi_0)}+ \| \uu\|_{L^2(0,T;L^2(\Omega))}\big)^2,
\end{equation}
which entails that 
\begin{equation}
\label{E-en3}
\| \phi\|_{L^2(0,T;H^2(\Omega))}
\leq C_\alpha \left( 1+ \| \nabla \phi_0\|_{L^2}+
\sqrt{1+T}\left(  \sqrt{E_{\text{free}}(\phi_0)}+ \| \uu\|_{L^2(0,T;L^2(\Omega))}\right) \right).
\end{equation}
Next, we control the total mass of the chemical potential. 
Arguing as for the Cahn-Hilliard equation, we multiply \eqref{vCH}$_2$ by $\phi-\overline{\phi}$ and integrate over $\Omega$. We find
\begin{align*}
\int_{\Omega} |\nabla \phi|^2 \, \d x+ \int_{\Omega} F'(\phi)(\phi-\overline{\phi}) \, \d x
= \int_{\Omega} \mu (\phi-\overline{\phi})\, \d x + \theta_0 \| \phi-\overline{\phi}\|_{L^2}^2-
\alpha \int_{\Omega} \partial_t \phi (\phi-\overline{\phi}) \, \d x.
\end{align*}
By using the Poincar\'{e} inequality and \eqref{vCH-weak}$_1$, we find
$$
\int_{\Omega} F'(\phi)(\phi-\overline{\phi}) \, \d x
\leq C_\alpha \left( 1+\| \nabla \mu\|_{L^2}+  \| \partial_t \phi\|_{L^2} \right), 
$$
for some $C_\alpha$ depending on $\Omega$, $\theta_0$ and $\alpha$.
We are now in position to control a full Sobolev norm of $\mu$. Thanks to  \cite[Proposition A.1]{MZ}, there exist two positive constants $C_1$, $C_2$ (only depending on $\overline{\phi_0}$) such that
$$
\int_{\Omega} |F'(\phi)|\, \d x \leq C_1 \int_{\Omega} F'(\phi)(\phi-\overline{\phi_0}) \, \d x +C_2,
$$ 
thus we infer that 
$$
\| F'(\phi)\|_{L^1} \leq
C_\alpha \left( 1+ \| \nabla \mu\|_{L^2}+  \| \partial_t \phi\|_{L^2} \right).
$$
Since $\overline{\mu}= \overline{F'(\phi)}-\theta_0 \overline{\phi_0}$, the above control yields
\begin{equation}
\label{mu-bar}
|\overline{\mu}| \leq
C_\alpha \left( 1+ \| \nabla \mu\|_{L^2}+  \| \partial_t \phi\|_{L^2} \right).
\end{equation}
As a result, it immediately follows that
\begin{equation}
\label{E-en4}
\| \mu\|_{L^2(0,T;H^1(\Omega))}\leq C_\alpha \left( \sqrt{T}+  \sqrt{E_{\text{free}}(\phi_0)}+ \| \uu\|_{L^2(0,T;L^2(\Omega))} \right).
\end{equation}
In addition, by using \eqref{vCH}$_1$ we observe that
\begin{align*}
\| \Delta \mu\|_{L^2} \leq \| \partial_t \phi\|_{L^2}+ \| \uu\|_{L^3}
 \| \nabla \phi\|_{L^6}.
\end{align*}
Then, combining the elliptic regularity with \eqref{E-en1} and \eqref{E-en3}, we find
\begin{equation}
\label{E-en5}
\| \mu\|_{L^2(0,T;H^2(\Omega))}\leq 
C \left( \alpha, E_{\text{free}}(\phi_0),T\right) \left( \left(1+
\| \uu\|_{L^\infty(0,T;L^3(\Omega))}\right) \left(1+ \| \uu\|_{L^2(0,T;L^2(\Omega))} \right)\right).
\end{equation}
By comparison in \eqref{vCH}$_2$, a similar estimate can be obtained for $F'(\phi)$ in $L^2(0,T;L^2(\Omega))$.
\smallskip

\noindent
\textbf{Uniqueness.}
Let $\phi_1$, $\phi_2$ be two weak solutions. We define the solutions difference $\psi= \phi_1-\phi_2$ which solves
$$
\partial_t \psi+ \uu \cdot \nabla \psi= \Delta \big( \alpha \partial_t \psi-\Delta \psi+ \Psi'(\phi_1)-\Psi'(\phi_2) \big) \quad \text{in } \Omega \times (0,T).
$$
Since $\overline{\psi}(t)=0$ for all $t\in [0,T]$, multiplying by $A^{-1} \psi$, where the operator $A$ is the Laplace operator $-\Delta$ with homogeneous Neumann boundary conditions, and integrating over $\Omega$, we obtain
\begin{align*}
\frac12 \ddt \bigg( \| \nabla A^{-1}\psi\|_{L^2}^2 + \alpha \| \psi\|_{L^2}^2 \bigg)+ \|\nabla \psi\|_{L^2}^2 \leq \int_{\Omega} \psi \uu \cdot \nabla A^{-1} \psi \, \d x+ \theta_0 \| \psi\|_{L^2}^2.
\end{align*}
Here we have used that $F'$ is a monotone function. Observing that
$$
\left| \int_{\Omega} \psi \uu \cdot \nabla A^{-1} \psi \, \d x \right|
\leq \| \psi\|_{L^2}\| \uu\|_{L^3} \|\nabla A^{-1}\psi \|_{L^6} \leq C \| \uu\|_{L^3} \| \psi\|_{L^2}^2,
$$ 
it is easily seen that
$$
\frac12 \ddt \bigg( \|\nabla A^{-1} \psi\|_{L^2}^2 + \alpha \| \psi\|_{L^2}^2 \bigg) \leq C \left(1+ \| \uu\|_{L^3}\right) \| \psi\|_{L^2}^2.
$$
An application of the Gronwall lemma yields
$$
\| \nabla A^{-1} \psi (t)\|_{L^2}^2 + \alpha \| \psi(t)\|_{L^2}^2
\leq \left( \| \nabla A^{-1} \psi(0)\|_{L^2}^2 + \alpha \| \psi(0)\|_{L^2}^2\right) \mathrm{e}^{C_\alpha \int_{0}^t (1+ \| \uu(\tau)\|_{L^3}) \, \d \tau}
$$
for all $t\in [0,T]$, which implies the uniqueness of the solution.
\smallskip

\noindent
\textbf{Regularity 1.} For $h \in (0,1)$, we define the notation $\partial_t^h f (\cdot, t)=\frac{1}{h} \big( f(\cdot, t+h)-f(\cdot, t) \big)$. We observe that $\phi \in C([0,T]; H^1(\Omega))$ and $\uu \in C([0,T]; L^1(\Omega))$, thereby we can extend both $\phi$ and $\uu$ on $[0,T+1]$ by $\phi(t)=\phi(T)$ and $\uu(t)=\uu(T)$ for $t\in (T,T+1]$. It follows from \eqref{vCH} that 
\begin{equation}
\label{Dif-eq}
\partial_t \partial_t^h \phi 
+ \partial_t^h \uu \cdot \nabla \phi (\cdot +h)+ \uu \cdot  \nabla \partial_t^h \phi
= \Delta( \varepsilon \partial_t \partial_t^h \phi-\Delta \partial_t^h \phi+ \partial_t^h \Psi'(\phi) ) \quad \text{in } \Omega \times (0,T).
\end{equation}
We multiply the above equation by $A^{-1}\partial_t^h \phi$ and integrate over $\Omega$. Exploiting the monotonicity of $F'$, the boundary condition of $\uu$ and the Agmon inequality \eqref{Agmon3d}, we obtain 
\begin{align*}
\frac12 \ddt &\bigg( \| \nabla A^{-1} \partial_t^h \phi\|_{L^2}^2 + \alpha \| \partial_t^h \phi\|_{L^2}^2 \bigg) +
\| \nabla \partial_t^h \phi\|_{L^2}^2\\
&\leq
\int_{\Omega} \phi(\cdot+h) \partial_t^h \uu \cdot \nabla A^{-1} \partial_t ^h \phi \, \d x+ \int_{\Omega} \partial_t^h \phi \, \uu \cdot \nabla A^{-1}\partial_t^h \phi \, \d x + \theta_0 \| \partial_t^h \phi\|_{L^2}^2\\
&\leq
\| \partial_t^h \uu\|_{L^1} 
\| \nabla A^{-1} \partial_t ^h \phi \|_{L^\infty} 
+\| \partial_t^h \phi\|_{L^2} \| \uu \|_{L^3} 
\| \nabla A^{-1}\partial_t^h \phi\|_{L^6} 
+ \theta_0 \| \partial_t^h \phi\|_{L^2}^2\\
&\leq 
C \| \partial_t^h \uu\|_{L^1} 
\| \partial_t ^h \phi \|_{L^2}^\frac12
\| \nabla \partial_t^h \phi\|_{L^2} 
+ C\left(1+ \| \uu \|_{L^3} \right)  \| \partial_t^h \phi\|_{L^2}^2\\
&\leq \frac12 \| \nabla \partial_t^h \phi \|_{L^2}^2+
C\| \partial_t^h \uu\|_{L^1}^\frac43 \left(1+\| \partial_t^h \phi\|_{L^2}^2\right) + C\left(1+ \| \uu \|_{L^3}\right)  \| \partial_t^h \phi\|_{L^2}^2.
\end{align*}
The Gronwall lemma entails
\begin{equation}
\label{diff-est}
\begin{split}
 &\alpha \| \partial_t^h \phi (t)\|_{L^2}^2
 + \int_0^t \| \nabla \partial_t^h \phi (\tau)\|_{L^2}^2 \, \d \tau\\
 &\quad \leq \left( \| \nabla A^{-1} \partial_t^h \phi(0)\|_{L^2}^2 + \alpha \| \partial_t^h \phi(0)\|_{L^2}^2 + C\int_0^t \| \partial_t^h \uu(\tau)\|_{L^1}^\frac43 \, \d \tau \right) \mathrm{e}^{\int_0^t g(\tau) \, \d \tau}
 \end{split}
\end{equation}
for all $t\in [0,T]$, where $g(\tau)= C_\alpha \left( 1+ \| \uu\|_{L^3}+ \| \partial_t^h\uu\|_{L^1}^\frac43\right)$. 
In order to control the right-hand side, we compute
\begin{align*}
&\frac12 \ddt \bigg( \| \nabla A^{-1}(\phi-\phi_0)\|_{L^2}^2+ \alpha \| \phi-\phi_0\|_{L^2}^2 \bigg)
%= (\partial_t \phi, A^{-1}(\phi-\phi_0))+ \alpha (\partial_t \phi, \phi-\phi_0)\\
=(\alpha \partial_t \phi- \mu, \phi-\phi_0)+ (\phi \,\uu, \nabla A^{-1}(\phi-\phi_0))\\
&\quad = (\Delta \phi-\Psi'(\phi), \phi-\phi_0) + (\phi \, \uu, \nabla A^{-1}(\phi-\phi_0))\\
&\quad = \underbrace{(\Delta(\phi-\phi_0)-(F'(\phi-F'(\phi_0)),\phi-\phi_0)}_{\leq 0}
+(\Delta \phi_0- F'(\phi_0), \phi-\phi_0)+ \theta_0 (\phi,\phi-\phi_0)
\\
&\quad \quad + (\phi \, \uu, \nabla A^{-1}(\phi-\phi_0)).
\end{align*}
Therefore, we have
\begin{equation*}
\frac12 \ddt \bigg( \| \nabla A^{-1}(\phi-\phi_0)\|_{L^2}^2+ \alpha \| \phi-\phi_0\|_{L^2}^2\bigg)
\leq C\big(1+ \| \Delta \phi_0- F'(\phi_0)\|_{L^2}+ \| \uu\|_{L^2}\big) \| \phi-\phi_0\|_{L^2}.
\end{equation*}
Thanks to \cite[Lemma 4.1, Chap. IV]{SHOW1997}, we arrive at
\begin{align*}
\| \nabla A^{-1}(\phi(t)-\phi_0)\|_{L^2}^2+ \alpha \| \phi(t)-\phi_0\|_{L^2}^2\leq \bigg( C_\alpha \big(1+ \| \Delta \phi_0- F'(\phi_0)\|_{L^2} \big) t+ C_\alpha \int_0^t \| \uu (\tau)\|_{L^2} \, \d \tau \bigg)^2
\end{align*}
for all $t\in [0,T]$. By choosing $t=h$, we deduce that
\begin{equation}
\label{in-con}
\begin{split}
\| \nabla A^{-1} \partial_t^h \phi(0)\|_{L^2}^2 &+ \alpha \| \partial_t^h \phi(0)\|_{L^2}^2 
\leq C_\alpha \left(1+ \| \Delta \phi_0- F'(\phi_0)\|_{L^2}^2 + \| \uu\|_{L^\infty(0,T;L^2(\Omega))}^2 \right).
\end{split}
\end{equation}
Since $\|\partial_t^h \uu\|_{L^\frac43(0,T;L^1(\Omega))}\leq \| \partial_t \uu\|_{L^\frac43(0,T; L^1(\Omega))}$, by combining \eqref{diff-est} and \eqref{in-con}, we obtain
\begin{equation}
\begin{split} 
&\alpha \| \partial_t^h \phi (t)\|_{L^2}^2 + \int_0^t \| \nabla \partial_t^h \phi (\tau)\|_{L^2}^2 \, \d \tau\\
&\leq C_\alpha \left(  1+ \| \Delta \phi_0- F'(\phi_0)\|_{L^2}^2 + \| \uu\|_{L^\infty(0,T;L^2(\Omega))}^2) + \| \partial_t \uu\|_{L^\frac43(0,T; L^1(\Omega))}^\frac43 \right) \mathrm{e}^{G(T)},
\end{split}
\end{equation}
for all $t \in [0,T]$, where $G(T)= \int_0^T C_\alpha \left(1+ \| \uu(\tau)\|_{L^3} \right) \, \d \tau+ C_\alpha \int_0^{T} \| \partial_t \uu(\tau)\|_{L^1}^\frac43 \, \d \tau$.
In light of the convergence $\partial_t^h \phi \rightarrow \partial_t \phi$ in $L^2(0,T;L^2(\Omega))$ as $h \rightarrow 0$, we infer that 
\begin{equation}
\label{Reg1}
\| \partial_t \phi\|_{L^\infty(0,T;L^2(\Omega))}
+ \| \partial_t \phi\|_{L^2(0,T;H^1(\Omega))}
\leq C(\alpha, T, \| \Delta \phi_0-F'(\phi_0)\|_{L^2}, \|\uu\|_{X_T}),
\end{equation}
where $X_T=L^\infty(0,T;L^3(\Omega))\cap W^{1,\frac43}(0,T;L^1(\Omega))$. Next, we derive further regularity properties on $\phi$ and $\mu$. By the incompressibility constraint, we recall that $\| \nabla \mu\|_{L^2}\leq C \left(\| \partial_t \phi\|_{L^2}+ \| \uu \|_{L^2} \right)$. Then, thanks to \eqref{mu-bar} and \eqref{Reg1}, we easily have
\begin{equation}
\| \mu\|_{L^\infty(0,T;H^1(\Omega))}\leq C\left( \alpha, T, \| \Delta \phi_0-F'(\phi_0)\|_{L^2(\Omega)}, \|\uu\|_{X_T}\right).
\end{equation}
As a consequence, by \cite[Theorem A.1]{GMT2019} we get
\begin{equation}
\label{phi-H2-vch}
\| \phi\|_{L^\infty(0,T;H^2(\Omega))}+ \| F'(\phi)\|_{L^\infty(0,T;L^2(\Omega))}\leq C \left( \alpha, T, \| \Delta \phi_0-F'(\phi_0)\|_{L^2(\Omega)}, \|\uu\|_{X_T}\right).
\end{equation}
Finally, since $\uu \in L^\infty(0,T;L^3(\Omega))$ and $\nabla \phi \in L^\infty(0,T;L^6(\Omega))$, by comparison in \eqref{vCH}$_1$, we also find  
\begin{equation}
\label{mu-H2-vch}
\| \mu\|_{L^\infty(0,T;H^2(\Omega))}\leq C \left( \alpha, T, \| \Delta \phi_0-F'(\phi_0)\|_{L^2(\Omega)}, \|\uu\|_{X_T} \right).
\end{equation}

\noindent
\textbf{Regularity 2.} 
Let us now write \eqref{vCH}$_2$ as follows
\begin{equation}
\alpha \partial_t \phi - \Delta \phi +F'(\phi)= h \quad \text{in } \Omega \times (0,T),
\end{equation}
where $h=\mu+ \theta_0 \phi$. Thanks to \eqref{mu-H2-vch}, 
$h\in L^\infty(0,T; L^\infty(\Omega))$. Next, we consider the ODEs problems
\begin{equation}
\begin{cases}
\alpha \partial_t U +F'(U)= \overline{H}, \\
U(0)= 1-\delta_0
\end{cases}
\quad 
\begin{cases}
\alpha \partial_t V +F'(V)= \underline{H}, \\
V(0)= -1+\delta_0
\end{cases}
\quad \text{in } (0,T),
\end{equation}
where $\overline{H}=\| h\|_{L^\infty}$ and $\underline{H}= -\| h\|_{L^\infty}$.
It is not difficult to show that there exist two unique solutions $U, V \in C([0,T])$ with $U_t, V_t \in L^\infty(0,T)$. In particular, since $\lim_{s\rightarrow \pm 1} F'(s)= \pm \infty$ and $\overline{H}, \underline{H} \in L^\infty(0,T)$, a simple comparison argument entails that there exists $\delta>0$ such that
$$
-1+\delta \leq V(t)\leq U(t) \leq 1-\delta, \quad \forall \, t \in [0,T].
$$
More precisely, it can be checked that $1-\delta \leq \max \lbrace 1-\delta_0, (F')^{-1}(\|\overline{H}\|_{L^\infty(0,T)})\rbrace$. We are left to show that $V(t)\leq \phi(x,t)\leq U(t)$ in $\Omega \times [0,T]$. To this aim, we use the Stampacchia method. We define $w=\phi-U$ and we consider the problem 
\begin{equation}
\label{diff-sep}
\begin{cases}
\alpha \partial_t w + \uu\cdot \nabla \phi - \Delta \phi +
F'(\phi)-F'(U)= h-\overline{H} \quad &\text{in } \Omega \times (0,T),\\
w(0)= \phi_0-1+\delta_0 \quad &\text{in } \Omega. 
\end{cases}
\end{equation}
Multiplying the equation by $w^+=\max \lbrace \phi-U, 0 \rbrace$ and integrating over $\Omega$, and using that $\nabla \phi= \nabla w^+$ on the set $\lbrace x\in \Omega: \phi \leq U \rbrace$, we find
$$
\frac{\alpha}{2}\ddt \| w^+\|_{L^2}^2
+ \int_{\Omega} (\uu\cdot \nabla w^+) w^+ \, \d x + \| \nabla w^+\|_{L^2}^2
+ \int_{\Omega} (F'(\phi)-F'(U)) w^+ \, \d x= \int_{\Omega} (h-\overline{H})w^+\, \d x.
$$ 
By the monotonicity of $F'$, it follows that
$$
\ddt \| w^+\|_{L^2}^2 \leq 0 \quad \Rightarrow \quad
\| w^+(t)\|_{L^2}^2\leq \| w^+(0)\|_{L^2}^2 =0, \quad \forall \, t \in [0,T],
$$
which, in turn, gives the desired result, namely $\phi(x,t)\leq U(t)$ in $\Omega \times [0,T]$. A similar argument entails that $V(t)\leq \phi(x,t)$ in $\Omega \times [0,T]$.
Therefore, we obtain by continuity the separation property 
\begin{equation}
\label{SP-2}
\max_{(x,t)\in \overline{\Omega} \times [0,T]} |\phi(x,t)|\leq 1-\delta.
\end{equation}
As a consequence, it follows from \eqref{phi-H2-vch} that $\Psi'(\phi)\in L^\infty(0,T; H^1(\Omega))$. Then, we deduce by comparison in \eqref{vCH}$_2$ and by elliptic regularity that 
$$
\| \phi\|_{L^2(0,T;H^3(\Omega))}
\leq C\left( \alpha, T, \delta, \| \Delta \phi_0-F'(\phi_0)\|_{L^2}, \|\uu\|_{X_T} \right).
$$

\noindent
\textbf{Regularity 3.} Thanks to the above regularity, we rewrite \eqref{Dif-eq} as follows
\begin{equation}
\label{test-pt}
\int_{\Omega} \partial_t \partial_t^h \phi \, v + \alpha \nabla \partial_t \partial_t^h \phi \cdot \nabla v \, \d x + \int_{\Omega} \partial_t^h ( \uu\cdot \nabla \phi ) v \, \d x=
\int_{\Omega} ( \nabla \Delta \partial_t^h \phi - \nabla \partial_t^h \Psi'(\phi)) \cdot \nabla v \,\d x 
\end{equation}
for all $v \in H^1(\Omega)$. Taking $v= \partial_t^h \phi$ and exploiting the boundary conditions of $\phi$ and $\uu$, we find
\begin{align*}
&\frac12 \ddt \bigg( \| \partial_t^h \phi\|_{L^2}^2+ \alpha \| \nabla \partial_t^h \phi\|_{L^2}^2 \bigg) + \int_{\Omega} |\Delta \partial_t^h \phi|^2 \, \d x \\
&\quad 
= \int_{\Omega} \partial_t^h(\uu \phi) \cdot \nabla \partial_t^h \phi \, \d x
+ \int_{\Omega} \partial_t^h F'(\phi) \Delta \partial_t^h \phi \, \d x +\theta_0 \| \nabla \partial_t^h \phi\|_{L^2}^2\\
&\quad \leq \| \partial_t^h \uu\|_{L^\frac65} \| \nabla \partial_t^h \phi\|_{L^6 }+ \| \uu\|_{L^3 }\| \partial_t^h \phi\|_{L^6 } 
\| \nabla \partial_t^h \phi\|_{L^2} + C \| \partial_t^h \phi\|_{L^2} \| \Delta \partial_t^h \phi\|_{L^2}+\theta_0 \| \nabla \partial_t^h \phi\|_{L^2(\Omega)}^2\\
&\quad \leq \frac12 \| \Delta \partial_t^h \phi\|_{L^2}^2+
C\| \partial_t^h \uu\|_{L^\frac65}^2+C \left(1+ \| \uu\|_{L^3}\right) \| \nabla \partial_t^h \phi\|_{L^2}^2 +C \|\partial_t^h \phi \|_{L^2}^2.  
\end{align*}
Here we have used the separation property \eqref{SP-2} and the inequality 
$\| \partial_t^h \phi\|_{H^2}\leq C \| \Delta \partial_t^h \phi \|_{L^2}$. Then, we infer from the Gronwall lemma that
\begin{equation}
\label{Reg2}
\begin{split}
&\| \partial_t^h \phi(t)\|_{L^2}^2+ \alpha \| \nabla \partial_t^h \phi (t)\|_{L^2}^2 + \int_0^t \| \Delta \partial_t^h \phi(\tau)\|_{L^2}^2 \, \d \tau\\
&\quad \leq \left( \| \partial_t^h \phi(0)\|_{L^2}^2+ \alpha \| \nabla \partial_t^h \phi (0)\|_{L^2}^2+ C \int_0^t \| \partial_t^h \uu(\tau)\|_{L^\frac65}^2 \,\d \tau \right) \mathrm{e}^{\widetilde{G}(T)}
\end{split}
\end{equation}
for all $t\in [0,T]$, where $\widetilde{G}(T)= C_\alpha \int_0^T (1+\| \uu(\tau)\|_{L^3})\, \d \tau $.
Since $\partial_\n \phi_0=0$ on $\partial \Omega$ by assumption, we observe that
\begin{align*}
&\frac12 \ddt \bigg( \| \phi-\phi_0\|_{L^2}^2+ \alpha \| \nabla (\phi-\phi_0)\|_{L^2}^2 \bigg) \\
& \quad = \int_{\Omega} \phi \, \uu \cdot \nabla (\phi-\phi_0) \, \d x
+ \int_{\Omega} \nabla (\Delta \phi- F'(\phi)+ \theta_0 \phi) \cdot \nabla (\phi-\phi_0) \, \d x\\
& \quad = \int_{\Omega} \phi \, \uu \cdot \nabla (\phi-\phi_0) \, \d x 
- \| \Delta(\phi-\phi_0)\|_{L^2}^2
+ \int_{\Omega} \nabla \Delta \phi_0 \cdot \nabla (\phi-\phi_0) \, \d x\\
&\quad \quad
+\int_{\Omega} \nabla (-F'(\phi)+\theta_0 \phi) \cdot \nabla (\phi-\phi_0) \, \d x.
\end{align*}
Thus, we obtain
$$
\frac12 \ddt \bigg( \| \phi-\phi_0\|_{L^2}^2+ \alpha \| \nabla (\phi-\phi_0)\|_{L^2}^2 \bigg)
\leq C\big(1+ \| \uu\|_{L^2} + \| \phi_0\|_{H^3} \big) \| \nabla (\phi-\phi_0)\|_{L^2}.
$$
By using \cite[Lemma 4.1, Chap. IV]{SHOW1997} and taking $t=h$, we arrive at 
\begin{equation}
\label{dato-H1}
\|\partial_t^h \phi(0) \|_{L^2}^2 +\alpha \| \nabla \partial_t^h \phi (0)\|_{L^2}^2 
\leq C_\alpha \left( 1+ \| \phi_0\|_{H^3}^2 + \| \uu\|_{L^\infty(0,T;L^2(\Omega))}^2\right).
\end{equation}
Combining the above inequality with \eqref{Reg2}, we are led to 
\begin{align*}
&\| \partial_t^h \phi(t)\|_{L^2}^2+ \alpha \| \nabla \partial_t^h \phi (t)\|_{L^2}^2 + \int_0^t \| \Delta \partial_t^h \phi(\tau)\|_{L^2}^2 \, \d \tau\\
&\quad \leq C_\alpha \left( 1+ \| \phi_0\|_{H^3}^2+ \| \uu\|_{L^\infty(0,T;L^2(\Omega))}^2 +\| \partial_t \uu\|_{L^2(0,T+1;L^\frac65(\Omega))}^2 \right) \mathrm{e}^{C\int_0^t (1+\| \uu(\tau)\|_{L^3})\, \d \tau}
\end{align*}
for all $t\in [0,T]$, which, in turn, implies 
\begin{equation}
\label{phi_t-vch}
\| \partial_t \phi\|_{L^\infty(0,T;H^1(\Omega))}+ \| \partial_t \phi\|_{L^2(0,T;H^2(\Omega))} \leq C(\alpha, T, \delta, \|\phi_0\|_{H^3}, \| \uu\|_{Y_T}),
\end{equation}
where $Y_T= L^\infty(0,T;L^3(\Omega))\cap W^{1,2}(0,T;L^\frac65(\Omega))$.
As an immediate consequence, in light of \eqref{phi-H2-vch}, \eqref{mu-H2-vch} and \eqref{SP-2}, we infer by comparison in \eqref{vCH}$_2$ that 
\begin{equation}
\| \phi\|_{L^\infty(0,T;H^3(\Omega))}+ \| \phi\|_{L^2(0,T;H^4(\Omega))}
\leq C(\alpha, T, \delta, \|\phi_0\|_{H^3(\Omega)}, \| \uu\|_{Y_T}),
\end{equation}
Next, we take $v= A^{-1} \partial_t^h \partial_t \phi$ in \eqref{test-pt}. 
Exploiting \eqref{SP-2} and \eqref{phi_t-vch}, we obtain
\begin{align*}
\frac12 \ddt \| \nabla \partial_t^h \phi\|_{L^2}^2 
&+ \| \nabla A^{-1}\partial_t^h \partial_t \phi\|_{L^2}^2
+\alpha \| \partial_t^h \partial_t \phi\|_{L^2}^2 \\
&\quad \leq \int_{\Omega} \partial_t^h (\phi \uu)\cdot \nabla A^{-1}\partial_t^h \partial_t \phi \,\d x - \int_{\Omega} \partial_t^h \Psi'(\phi) \partial_t^h \partial_t \phi \, \d x\\
&\quad \leq C \| \partial_t \uu\|_{L^\frac65} \|\partial_t^h \partial_t\phi \|_{L^2}+ C \| \uu\|_{L^3} \| \partial_t \phi\|_{L^2} \| \nabla A^{-1}\partial_t^h \partial_t \phi\|_{L^6} +C \| \partial_t^h \phi \|_{L^2} 
\| \partial_t^h \partial_t \phi \|_{L^2}\\
&\quad \leq \frac12\|\partial_t^h \partial_t\phi \|_{L^2}^2+
 C \left( 1+\| \partial_t \uu\|_{L^\frac65}^2 + \| \uu\|_{L^3}^2 \right).
\end{align*}
By recalling \eqref{dato-H1}, the Gronwall lemma entails
\begin{equation}
\int_0^T \| \partial_t^h \partial_t \phi\|_{L^2}^2 \, \d \tau
\leq C(\alpha, T, \delta, \|\phi_0\|_{H^3}, \| \uu\|_{Y_T}),
\end{equation}
which, in turn, gives that there exists $\partial_t^2 \phi \in  L^2(0,T;L^2(\Omega))$ such that 
$$
\| \partial_t^2 \phi\|_{L^2(0,T;L^2(\Omega))}\leq C(\alpha, T, \delta, \|\phi_0\|_{H^3}, \| \uu\|_{Y_T}).
$$
Thus, by comparison in \eqref{vCH}, we conclude that there exists $\partial_t \mu \in L^2(0,T;L^2(\Omega))$ such that 
$$
\| \partial_t \mu \|_{L^2(0,T;L^2(\Omega))}
\leq C(\alpha, T, \delta, \|\phi_0\|_{H^3(\Omega)}, \| \uu\|_{Y_T}).
$$ 
The proof is complete.
\end{proof}

\end{document}